\documentclass[a4paper,12pt]{article}

\usepackage[T1]{fontenc}
\usepackage[utf8]{inputenc}
\usepackage{graphicx}

\newenvironment{gras}
  {\fontseries{b}\selectfont}{}

\usepackage{amsmath}
\usepackage{amsfonts}
\usepackage{amssymb}
\usepackage{amsthm}
\usepackage{enumerate}

\newtheorem{theor}{Theorem}[section]
\newtheorem{defin}[theor]{Definition}
\newtheorem{prop}[theor]{Proposition}
\newtheorem{lem}[theor]{Lemma}
\newtheorem{cor}[theor]{Corollary}
\newtheorem{rem}[theor]{Remark}

\numberwithin{equation}{section}

\makeatletter
\def\maketitle{
	\begin{titlepage}
	\null\vfill
	\begin{center}
		{\LARGE \@title}
		\vskip 1cm
		{\Large \@author}
		\vskip 1cm
	\end{center}
	\vfill\null
	\end{titlepage}
}
\makeatother

\newcommand{\N}{\mathbb{N}}
\newcommand{\Z}{\mathbb{Z}}

\newcommand{\R}{\mathbb{R}}
\newcommand{\C}{\mathbb{C}}

\newcommand{\Bcal}{\mathcal{B}}
\newcommand{\Ccal}{\mathcal{C}}
\newcommand{\Fcal}{\mathcal{F}}
\newcommand{\Hcal}{\mathcal{H}}
\newcommand{\Scal}{\mathcal{S}}
\newcommand{\Tcal}{\mathcal{T}}
\newcommand{\Lbb}{\mathbb{L}}

\newcommand{\V}{\mathbb{V}}

\newcommand{\BV}{\mathbb{BV}}
\newcommand{\norm}[2]{\left\| #1 \right\|_{#2}}

\newcommand{\Lcal}{\mathcal{L}}
\newcommand{\dd}{\;{\rm d}}

\newcommand{\st}{\;|\;}
\newcommand{\ii}{\mathbf{i}}
\newcommand{\transposee}[1]{{\vphantom{#1}}^{\mathit t}{#1}}
\DeclareMathOperator{\LL}{Leb}
\DeclareMathOperator{\Lip}{Lip}

\DeclareMathOperator{\dL}{dLeb}
\DeclareMathOperator{\Supp}{Supp}
\DeclareMathOperator{\Esup}{Esup}
\DeclareMathOperator{\Einf}{Einf}
\DeclareMathOperator{\osc}{osc}
\DeclareMathOperator{\Card}{Card}

\usepackage[top=3cm, bottom=2cm, left=3cm, right=3cm]{geometry}

\title{A spectral gap for transfer operators of piecewise expanding maps}

\author{Damien THOMINE}

\begin{document}

\maketitle

\thispagestyle{empty}

\newpage

\tableofcontents

\newpage

\section*{Introduction}
\addcontentsline{toc}{section}{Introduction}

An important task in dynamical systems, and the whole point 
of ergodic theory, is to study the measures invariant under some 
transformation. We will focus here on physical measures, which are 
empirical measures for a set of initial conditions of non-zero Lebesgue measure (see Subsection~\ref{subsec:physical} for a definition of empirical mesures). To obtain non-trivial 
results, some assumptions must be made on those transformations, for 
instance on their smoothness. We shall study here maps which are piecewise differentiable, 
with H\"{o}lder derivative, and uniformly expanding (that is, with positive and bounded away from zero 
Lyapunov exponents) on some compact space.

\smallskip

In one dimension, a most interesting result was found in 1973 by A.~Lasota 
and J.A. Yorke \cite{Lasota}: a piecewise twice differentiable uniformly expanding 
map on an interval admits a finite number of physical invariant 
measures, whose densities have bounded variation, and (up to taking an 
iterate of the transformation) they are mixing. This is obtained 
via the study of the Perron-Frobenius operator on the space of functions 
with bounded variation: the fixed points of this operator are densities of 
invariant measures, and the existence of a spectral gap provides the property 
of mixing at an exponential rate. To sum it up, the quasi-compactness of the 
Perron-Frobenius operator, along with some classical arguments, is enough to 
deduce many valuable ergodic properties of the system.

\smallskip

Problems arise when we weaken those assumptions. A possible 
generalization is to obtain results in higher dimensions, 
where general answers were obtained only recently, despite some 
early results on particular systems~\cite{Keller1}. 
Actually, a naive generalization of the theorem proved by A.~Lasota and J.A.~Yorke does not hold in 
dimension $2$ or higher, as is shown by the counter-examples of M.~Tsujii~\cite{Tsujii} 
and J.~Buzzi~\cite{Buzzi1}: some additional assumptions must be made. 
A first strategy is to require a lower bound on the angles made by the discontinuities; 
P.~G\'{o}ra and A.~Boyarski~\cite{Gora1} made a successful attempt in 1989, although 
the sufficient condition they found on the expanding rate is far from optimum. 
Another strategy is to study the combinatorial complexity, i.e. the way the space is cut by 
the discontinuities and the number of times they overlap; it was used for instance in 
2000 by J.W.~Cowieson~\cite{Cowieson1} and B.~Saussol~\cite{Saussol}, and we will follow this approach.

\smallskip

In all those studies, a major problem is to find a suitable space of functions 
on which the Perron-Frobenius operator may act. For piecewise smooth
maps, the function spaces we use should show some regularity (otherwise, no spectral gap will be found), 
but not too much, since discontinuities must be allowed. For piecewise twice differentiable maps, 
the space of functions with bounded variation is suitable, as was shown by 
J.W.~Cowieson~\cite{Cowieson1} (although he proved only the existence of invariant 
measures, and not stronger properties). For less smooth maps, B.~Saussol worked 
with spaces of functions with bounded oscillation~\cite{Saussol}. Following the 
method used in \cite{BaladiHyp1}, we will work here mainly with Sobolev spaces.

\smallskip

The approach followed in this article has many advantages. We work with usual 
and well-known spaces, instead of ad hoc function spaces. 
Moreover, the dynamic does not need to be studied in detail: this gives a very robust method, 
which can be more or less easily adapted to prove similar results for other 
function spaces or dynamical systems, since it relies only on the core properties 
of the function spaces.

\smallskip

This paper is the outcome of a master's thesis, supervised by V.~Baladi. The main goal is to simplify 
the proofs of the previous paper from V.~Baladi and S.~Gouëzel \cite{BaladiHyp1},
restricting ourselves to the study of expanding maps. 
Most of the setting and proofs in sections $1$ to $6$ are directly adapted from that work 
(a notable exception being Lemma~\ref{lem:truncate}, corresponding to Lemma~31 in 
\cite{BaladiHyp1} but whose proof is adapted from another article \cite{BaladiHyp2}). 
There are multiple gains: the setting is simpler, the function spaces 
used more standard (classical Sobolev spaces) and most of the proofs much 
shorter (some of which are even found in the literature). We will also deal with
a limit case, when the transformation is piecewise differentiable with a Lipschitz 
derivative, and prove a new result involving functions with bounded variation - 
basically a stronger version of J.W.~Cowieson's theorem~\cite{Cowieson1}. 

\smallskip

The definitions and setting are explained in 
Section~1, as well as the main results: a bound on the essential spectral radius of 
transfer operators, Theorem~\ref{theor:spectral_gap}, and its consequence on the existence of 
finitely many mixing physical measures with densities in appropriate Sobolev spaces, and 
on the rate of mixing for H\"{o}lder test functions (Corollary~\ref{cor:spectral_gap} and 
Theorem~\ref{thm:ExistSRB}). We shall also present an application of those result on a class of piecewise affine maps.
Section~\ref{sec:sobolev} presents extensively the Sobolev spaces and their basic properties, as well as the space of functions with bounded variation.
Section~\ref{sec:inequalities} contains the main lemmas, which exploit the properties of 
the Sobolev spaces; they are used in Section~\ref{sec:main} to prove Theorem~\ref{theor:spectral_gap}. 
Section~\ref{sec:physical} contains the proof of Theorem~\ref{thm:ExistSRB}. Finally, in Section~\ref{sec:variation} 
we will deal with a limit case, where the transformation is piecewise differentiable with Lipschitz 
derivative and one works on functions with bounded variation, and a discussion on the 
previous similar results by J.W.~Cowieson~\cite{Cowieson1} and B.~Saussol~\cite{Saussol}. 
This last section also highlights the fact that our main theorems can be quite 
easily adapted to non-Sobolev function spaces.


\section{Setting and results}
\label{sec:setting}

If $B$  is a Banach space, we denote the norm of an element $f$ of $B$ by
$\norm{f}{B}$. In this paper, a map defined on a closed
subset of a manifold is said to be $\Ccal^k$ or $\Ccal^\infty$ if it
admits an extension to a neighborhood of this closed subset,
which is $\Ccal^k$ or $\Ccal^\infty$ in the usual sense. For $\alpha \in (0,1)$, 
a map is said to be $\Ccal^\alpha$ if it is $\alpha$-H\"{o}lder, $\Ccal^{1+\alpha}$ 
if it is $\Ccal^1$ with $\alpha$-H\"{o}lder derivative, and $\Ccal^{1+\Lip}$ if it 
is $\Ccal^1$ with Lipschitz derivative.



Let $X$ be a Riemannian manifold of dimension $d$, and let
$X_0$ be a compact subset of $X$.
We call $\Ccal^1$ hypersurface with boundary a codimension-one
$\Ccal^1$ submanifold of $X$ with boundary (i.e., every point of
this set has a neighborhood diffeomorphic either to $\R^{d-1}$
or $\R^{d-2}\times [0,+\infty)$).

\begin{defin}[Piecewise $\Ccal^{1+\alpha}$ and $\Ccal^{1+\Lip}$ expanding maps]\quad
\label{def:piecewise_expanding}

For $\alpha>0$, we say that a map $T: X_0
\to X_0$ is a piecewise $\Ccal^{1+\alpha}$ expanding map 
(respectively piecewise $\Ccal^{1+\Lip}$ expanding map) if:
\begin{itemize}
\item There exists a finite number of disjoint open subsets
$O_1,\dots,O_I$ of $X_0$, covering Lebesgue-almost all
$X_0$, whose boundaries are unions of finitely many compact
$\Ccal^1$ hypersurfaces with boundary.
\item For $1 \leq i \leq I$, there exists a $\Ccal^{1+\alpha}$ (respectively $\Ccal^{1+\Lip}$) map $T_i$
defined on a neighborhood of $\overline{O_i}$, which is a
diffeomorphism onto its image, such that $T$ coincides with
$T_i$ on $O_i$.
\item For any $x \in X_0$ and whenever $T_i (x)$ is defined, 
$\displaystyle \lambda (x) := \inf_{v\in \Tcal_x} \frac{|DT_i (x)
v|}{|v|} > 1$.
\end{itemize}
\end{defin}

\smallskip

Since we choose to study the asymptotical combinatorial complexity, we have to 
quantify it. It will be done the following way.

Let choose a non-zero integer $n$.

Let $\ii=(i_0,\dots,i_{n-1})\in \{1,\dots,I\}^n$. We define
inductively sets $O_\ii$ by $O_{(i_0)}=O_i$, and
  \begin{equation}
  O_{(i_0,\dots,i_{n-1})}=\{x\in O_{i_0} \st T_{i_0}x\in
  O_{(i_1,\dots,i_{n-1})}\}.
  \end{equation}
Let also $T_\ii^n=T_{i_{n-1}} \circ \dots \circ T_{i_0}$; it is
defined on a neighborhood of $O_\ii$.

For any $x \in X_0$ and whenever $T_\ii^n (x)$ is defined, we denote :

\[
\lambda_n (x) := \inf_{v\in \Tcal_x} \frac{|DT_\ii^n (x)v|}{|v|} > 1.
\]

We define the complexity at the beginning
  \begin{equation}\label{cpb}
  D^b_n=\max_{x\in X_0} \Card \{ \ii=(i_0,\dots,i_{n-1}) \st x \in
  \overline{O_{\ii}} \},
  \end{equation}
and the complexity at the end
  \begin{equation}\label{cpe}
  D^e_n=\max_{x\in X_0} \Card \{ \ii=(i_0,\dots,i_{n-1}) \st x \in
  \overline{T^n(O_{\ii})} \}.
  \end{equation}

For $T(x)=2x$ mod $1$ on $[0,1]$ we have $D^e_n \geq 2^n$,
but fortunately this quantity plays no role
when $T$ is piecewise $\Ccal^{1+\Lip}$, where we 
can take $p$ as close to $1$ as needed in Corollary~\ref{cor:spectral_gap} (cf. Corollary~\ref{cor:stupid}), 
or work in the space of functions with bounded variation - see Theorem~\ref{theor:BVspectral_gap} or Cowieson's Theorem (referenced here as Theorem~\ref{theor:cowieson}).

For generic piecewise expanding maps, the complexity $D^b_n$ increases 
subexponentially, and therefore does not play an important role in the spectral 
formula \eqref{eq:rayon_spectral_essentiel} below (see \cite{Cowieson2}). 
However, some pathologic cases may arise when the iterates of the map $T$ 
cut some open set $\Omega$ too often, and then map it onto itself. In spite 
of the expansion, these numerous "cut-and-fold" may make the images of $\Omega$ 
arbitrarily small, so that the physical measure starting from any point of 
$\Omega$ can be, for instance, a Dirac measure. This phenomenon is central 
in the examples given by M.~Tsuji \cite{Tsujii} and J.~Buzzi \cite{Buzzi1}. 
A control on the combinatorial complexities ensures that the expansion do beat the cuts.


\subsection{Spectral results}

The results about the physical measures will be obtained 
through the study of transfer operators (or Perron-Frobenius operators); we now define them. 

\begin{defin}[Transfer operator]\quad

Let $\alpha >0$. For all $g : X_0 \to \C$ such that the restriction of $g$ 
to any $O_i$ is $\Ccal^\alpha$ (in the sense that it admits 
a $\Ccal^\alpha$ extension to some neighborhood of $\overline{O_i}$), 
we define the transfer operator $\Lcal_g$ on $\Lbb^\infty (X_0)$ by:

\begin{equation}
\Lcal_g u (x) = \sum_{y \in T^{-1}(\{x\})} g(y) u(y).
\end{equation}
\end{defin}

These operators will act on Sobolev spaces $\Hcal_p^t$, for $1 < p < \infty$ and $0<t<\min ( 1/p,\alpha )$, or on functions with bounded variation. 
As for now, we will just recall that the Sobolev spaces $H_p^t$ on $\R^d$ are 
defined as $\{ u \in \Lbb^p : \Fcal^{-1}((1+|\xi|^2)^{t/2} (\Fcal u) (\xi)) \in \Lbb^p\}$, 
with $\norm{u}{H_p^t} = \norm{\Fcal^{-1}((1+|\xi|^2)^{t/2} (\Fcal u) (\xi))}{\Lbb^p}$, 
where $\Fcal$ denotes the Fourier transform. 
Using charts, one can define a Sobolev space $\Hcal_p^t$ on $X$. 
Here, $t$ is an index for regularity; as $t$ increases, $\Hcal_p^t$ 
contains more regular functions. The Sobolev and "bounded variation" spaces will be defined precisely in Section~\ref{sec:sobolev}.

The first result is about the essential spectral radius of $\Lcal_g$ when acting on $\Hcal_p^t$; 
then, an application of this result  to $\Lcal_{1/|\det DT|}$, the Perron-Frobenius operator, will lead to the proof (under 
some conditions) of the existence of finitely many mixing physical measures.

\begin{theor}[Spectral theorem for piecewise $\Ccal^{1+\alpha}$ expanding maps]\quad
\label{theor:spectral_gap}

Let $\alpha \in (0,1]$, and let $T$ be a piecewise $\Ccal^{1+\alpha}$  
uniformly expanding map. Choose $p \in (1,+\infty)$, and 
$0 < t < \min (1/p,\alpha)$. Let $g : X_0 \mapsto \C$ be a function 
such that the restriction of $g$ to any $O_i$ admits a $\Ccal^\alpha$ 
extension to $\overline{O_i}$. Then $\Lcal_g$ acts continuously on $\Hcal_p^t$, 
where its essential spectral radius is at most

\begin{equation}
\label{eq:rayon_spectral_essentiel}
\lim_{n \rightarrow +\infty} (D_n^b)^{\frac{1}{p} \frac{1}{n}} 
(D_n^e)^{(1-\frac{1}{p}) \frac{1}{n}} \| g^{(n)} | \det DT^n|^\frac{1}{p} 
\lambda_n^{-t} \|_{\Lbb^\infty}^\frac{1}{n},
\end{equation}

where $\displaystyle g^{(n)} = \prod_{i=0}^{n-1} g \circ T^i$, 
and the limit exists by submultiplicativity.
\end{theor}

When we say that $\Lcal_g$ acts continuously on $\Hcal_p^t$,
we mean that, for any $u\in \Hcal_p^t \cap \Lbb^\infty(\LL)$, 
the function $\Lcal_g u$, which is essentially bounded, still belongs to $\Hcal_p^t$
and satisfies $\norm{\Lcal_g u}{\Hcal_p^t}\leq C
\norm{u}{\Hcal_p^t}$. Since $\Lbb^\infty$ is
dense in $\Hcal_p^t$ (by Theorem~3.2/2 in~\cite{Triebel3}), the operator
$\Lcal_g$ can be extended to a continuous operator on
$\Hcal_p^t$. The restriction $0<t<1/p$ is exactly designed so that
the space $H_p^t$ is stable under multiplication by
characteristic functions of nice sets (see Lemma~\ref{lem:multiplication_charfunction}). 
We also provide a version of this theorem for piecewise $\Ccal^{1+\Lip}$ maps, including piecewise $\Ccal^2$ maps. 
We cannot use the function space $H_1^1$ (which for 
instance contains only continuous functions in dimension $1$); 
we will instead use the space of functions with bounded variation $\BV$, 
which will be defined rigorously in Section~\ref{sec:sobolev}:

\begin{theor}[Spectral theorem for piecewise $\Ccal^{1+\Lip}$ expanding maps]\quad
\label{theor:BVspectral_gap}

Let be  $T$ a piecewise $\Ccal^{1+\Lip}$ expanding map. 
Let $g : X_0 \mapsto \C$ be a function such that the restriction 
of $g$ to any $O_i$ admits a Lipschitz extension to $\overline{O_i}$. 
Then $\Lcal_g$ acts continuously on $\BV (X_0)$, where its essential 
spectral radius is at most

\begin{equation}
\label{eq:BVrayon_spectral_essentiel}
\lim_{n \to +\infty} (D_n^b)^\frac{1}{n} 
\| g^{(n)} | \det DT^n| \lambda_n^{-1} \|_{\Lbb^\infty}^\frac{1}{n}.
\end{equation}
\end{theor}


\subsection{Physical measures}
\label{subsec:physical}

The empirical measure of $T$ with initial condition $x \in X_0$ is the weak limit, 
if it exists, of $1/n \sum_{k=0}^{n-1}\delta_{T^k x}$; such a measure is always $T$-invariant. 
A physical measure of $T$ is an invariant probability measures $\mu$ which is an 
empirical measure for a set of initial conditions of non-zero Lebesgue measure; the basin of a physical measure $\mu$ is the set of all initial conditions whose corresponding empirical measure is $\mu$. 

In the following and for any finite measure $\mu$ and integrable function $f$, 
the expression $\langle \mu, f \rangle$ will be used for the integral of $f$ against $\mu$. We may see any integrable function $u$ as the density (with respect to Lebesgue measure $\LL$) of some finite measure, and use $\langle u, f \rangle$ instead of $\langle u \dL, f \rangle$. For all $f \in \Lbb^\infty$ and $u \in \Hcal_p^t$, we have 
$\langle u, f \circ T \rangle = \langle \Lcal_{1/|\det DT|} u, f \rangle$; 
a nonnegative and nontrivial fixed point $u$ of $\Lcal_{1/|\det DT|}$ in 
$\Hcal_p^t$ then corresponds (up to normalization) to an invariant probability 
measure $\mu_u = u \dL$ whose density with respect to the Lebesgue measure is in $\Hcal_p^t$. By Birkhoff's theorem, when such a measure $\mu_u$ is ergodic, it is 
physical, since $1/n \sum_{k=0}^{n-1}\delta_{T^k x}$ converges weakly 
to $\mu_u$ for Lebesgue-almost every point in the support of $u$.

Theorem \ref{theor:spectral_gap} implies in this setting:

\begin{cor}\quad
\label{cor:spectral_gap}

Under the hypotheses of Theorem \ref{theor:spectral_gap}, assume that
\begin{equation}
\label{eq:good_condition}
\lim_{n \to +\infty}
(D_n^b)^{\frac{1}{p}\frac{1}{n}} \cdot
(D_n^e)^{(1-\frac{1}{p})\frac{1}{n}} \cdot
\norm{\lambda_n^{-t} |\det DT^n|^{\frac{1}{p}-1}}{\Lbb^\infty}^\frac{1}{n} <1.
\end{equation}
Then the essential spectral radius of $\Lcal_{1/|\det DT|}$ acting on
$\Hcal_p^t$ is strictly smaller than $1$.
\end{cor}

If \eqref{eq:good_condition} holds, we will be able to prove the existence 
of a spectral gap, i.e. of some $\tau < 1$ such that 
$\{ z \in Spec(\Lcal_{1/|\det DT|}): \tau < |z| <1 \} = \emptyset$, 
together with a nice repartition of the eigenvalues on the unit cicle. 
The number $\tau$ can be larger the essential spectral radius, as is shown for instance by G.~Keller and H.H.~Rugh 
in~\cite{Keller2} in a $1$-dimensional setting. 
This will imply the following:

\begin{theor}\quad
\label{thm:ExistSRB}

Under the assumptions of Theorem~\ref{theor:spectral_gap}, 
if \eqref{eq:good_condition} holds, then $T$ has a finite number 
of physical measures whose densities are in $\Hcal_p^t$, which are 
ergodic, and whose basins cover Lebesgue almost all $X_0$.
Moreover, if $\mu$ is one of these measures, there exist an
integer $k$ and a decomposition $\mu=\mu_1+\dots+\mu_k$ such
that $T$ sends $\mu_j$ to $\mu_{j+1}$ for $j\in \Z/k\Z$, and
the probability measures $k\mu_j$ are mixing at an exponential rate for
$T^k$ and $\alpha$-H\"{o}lder test functions.
\end{theor}

The above will be proved in Section~\ref{sec:physical}, although the 
arguments are classical. We shall also obtain its bounded variation counterpart:

\begin{theor}\quad
\label{thm:BVExistSRB}

Under the assumptions of Theorem~\ref{theor:BVspectral_gap}, 
assume that: 
\begin{equation}
\label{eq:BVgood_condition}
\lim_{n \to +\infty}
(D_n^b)^\frac{1}{n} \cdot
\norm{\lambda_n^{-1}}{\Lbb^\infty}^\frac{1}{n} <1.
\end{equation}
Then $T$ has a finite number of physical measures whose densities 
are in $\BV (X_0)$, which are ergodic, and whose basins cover Lebesgue almost all $X_0$.
Moreover, if $\mu$ is one of these measures, there exist an
integer $k$ and a decomposition $\mu=\mu_1+\dots+\mu_k$ such
that $T$ sends $\mu_j$ to $\mu_{j+1}$ for $j\in \Z/k\Z$, and
the probability measures $k\mu_j$ are mixing at an exponential rate for
$T^k$ and $\alpha$-H\"{o}lder test functions.
\end{theor}

\subsection{Piecewise affine maps}

We now describe an explicit class of maps for which the assumptions of the previous theorems 
are satisfied. Let $A_1,...,A_N$ be $d\times d$ matrices with no eigenvalue of modulus smaller or equal to
$1$, such that any two of these matrices commute. Let $X_0$ be a (non necessarily connected) polyhedral region of $\R^d$, 
and define a map $T$ on $X_0$ by cutting it into finitely many polyhedral subregions $O_1,\dots,O_N$, 
applying $A_i$ to each open set $O_i$, and then mapping $A_1 O_1,\dots, A_N O_N$ back into $X_0$ by 
translations. It is obviously piecewise $\Ccal^\infty$ 
and uniformly expanding. Let $\lambda$ be the lowest modulus of all the eigenvalues of all the matrices $A_i$; 
we get easily $\lambda_n \geq \lambda^n$.

\begin{prop}\quad
\label{prop:Affine}

Under those assumptions, the essential spectral radius of $\Lcal_{1/|\det DT|}$ acting on $\BV$ is at most $\lambda^{-1}$, 
thus strictly smaller than $1$. Therefore, $T$ satisfies the conclusions of Theorem~\ref{thm:BVExistSRB}.
\end{prop}

\begin{proof}\quad

Let $K$ be the total number of the sides of the polyhedra $O_i$. Around any point $x$, 
the boundaries of the sets $O_{(i_0,\dots,i_{n-1})}$ are preimages of theses sides by 
one of the maps $\displaystyle A_{i_0},\dots, \prod_{l=0}^{n-1} A_{i_l}$. Hence, there are 
at most $\displaystyle J(n)=K \sum_{k=1}^n \Card \{ \prod_{l=0}^{k-1} A_{i_l} : 
\{i_0,...,i_{k-1} \} \in \{1,..,N\}^k \}$ such preimages. Since all the $A_i$ 
commute, there are at most $k^N$ different maps which can be written as $\displaystyle \prod_{l=0}^{k-1} A_{i_l}$, 
and $J(n)$ grows polynomially. Following the claim p.105 in \cite{Buzzi3}, 
$D_n^b \leq 2 J(n)^d$. This quantity grows subexponentially.

By Theorem~\ref{theor:BVspectral_gap}, the essential spectral radius of $\Lcal_{1/|\det DT|}$ 
acting on $\BV$ is bounded by $\lambda^{-1}$.
\end{proof}

This proposition can also be deduced from Cowieson's theorem~\cite{Cowieson1} whenever the polyhedral domain is connected.


\section{Sobolev and bounded variation spaces}
\label{sec:sobolev}

It is now time to define precisely the function spaces we work with. 
First, we will present the Sobolev spaces $\Hcal_p^t$ ($p \in (0,1)$ and $0 \leq t$) 
and some of their properties. Since they have been thoroughly 
studied in the 1960s-1970s, we may exploit many ready-to-use results. 
They are spaces of $\Lbb^p$ functions which satisfy some regularity condition; 
since the transfer operator of expanding dynamics tends to improve 
the regularity, this may imply the existence of a spectral gap for 
these operators. We show two equivalent ways to define these spaces; 
the first is via the Fourier transform, the second via interpolation theory. 
Then, we will present the space of functions with bounded variation, which 
is nicer than the Sobolev space $\Hcal_1^1$ in our setting.


\subsection{Definition via Fourier transform}
\label{subsec:SobolevAndFourier}

\begin{defin}[Local spaces $H_p^t$]\quad
\label{def:space}

For $1<p<\infty$, $t \geq 0$, $\xi \in \R^d$, put $a_t (\xi) 
= (1+|\xi|^2)^{t/2}$. We define the space $H_p^t$ of functions in 
$\R^d$ as the subspace of functions $u \in \Lbb^p$ such that 
$\displaystyle \Fcal^{-1}(a_t \Fcal u)\in \Lbb^p$ with its canonical 
norm, i.e., the $\Lbb^p$ norm of the expression above.
\end{defin}

Since $a_t$ increases at infinity, the condition $\displaystyle 
\Fcal^{-1}(a_t \Fcal u)\in \Lbb^p$ can be understood as a condition 
on the decay at infinity of $\Fcal u$, hence on the regularity of $u$.

If $0 \leq t < \alpha$, we shall see that $H_p^t$ is invariant 
under composition by $\Ccal^{1+\alpha}$ diffeomorphisms: this is 
Lemma~\ref{lem:composition_diffeo}. Hence, we can glue such 
spaces locally together in appropriate coordinate patches, to 
define a space $\Hcal_p^t$ of functions on the manifold:

\begin{defin}[Sobolev spaces on $X_0$]\quad
\label{def:spaceX}

Let $0 \leq t < \alpha$. Fix  a finite number of $C^{1+\alpha}$
charts $\kappa_1,\dots,\kappa_J$ whose domains of definition
cover a compact neighborhood of $X_0$, and a partition of unity
$\rho_1,\dots,\rho_J$, such that the support of $\rho_j$ is
compactly contained in the domain of definition of $\kappa_j$,
and $\sum \rho_j=1$ on $X$. The space $\Hcal_p^t$ is then
the space of functions $u$ supported on $X_0$ such that $(\rho_j
u)\circ \kappa_j^{-1}$ belongs to $H_p^t$ for all $j$,
endowed with the norm
\begin{equation}
\label{defnorm}
\norm{u}{\Hcal_p^t}
= \sum_{j=1}^J \norm{ (\rho_j u) \circ \kappa_j^{-1}}{H_p^t}.
\end{equation}
\end{defin}

Changing the charts and the partition of unity gives
an equivalent norm on the same space of functions
by Lemma~\ref{lem:multiplication_smooth_functions}.
To fix ideas, we shall view the charts and partition of unity as fixed.

The Fourier transform approach also provides some effective theorems, 
for instance Fourier multipliers theorems (the following is e.g. Theorem~2.4/2 in \cite{Triebel3}):

\begin{theor}[Marcinkiewicz multiplier theorem]\quad
\label{thm:Marcinkiewicz}

Let $b \in \Ccal^d(\R^d)$ satisfy $|\zeta^\gamma
D^\gamma b(\zeta)| \leq K$ for all multi-indices $\gamma \in \{0,1\}^d$,
and all $\zeta\in \R^d$. For all $p \in (1,\infty)$, there
exists a constant $C_{p,d}$ which depends only on $p$ and $d$ such that, for any $u \in \Lbb^p$,

\begin{equation}
\norm{ \Fcal^{-1}( b \Fcal u)}{\Lbb^p} \leq C_{p,d} K \norm{u}{\Lbb^p}.
\end{equation}
\end{theor}


\subsection{Definition via interpolation theory}

Now, let us present the complex interpolation theory, developped by 
J.L.~Lions, A.P.~Calder\'{o}n and S.G.~Krejn (see e.g. §1.9 in \cite{Triebel1}).

A pair $(\Bcal_0, \Bcal_1)$ of Banach
spaces is called an interpolation couple if they are both
continuously embedded in a linear Hausdorff space $\Bcal$. For
any interpolation couple $(\Bcal_0, \Bcal_1)$, we let $L(\Bcal_0,
\Bcal_1)$ be the space of all linear operators $\Lcal$ mapping
$\Bcal_0+\Bcal_1$ to itself so that $\Lcal_{|\Bcal_j}$ is continuous
from $\Bcal_j$ to itself for $j=0,1$. For an interpolation couple
$(\Bcal_0, \Bcal_1)$ and $0 < \theta < 1$, we next define $[\Bcal_0,\Bcal_1]_\theta$ the 
complex interpolation space of parameter $\theta$.

Set $S= \{ z \in \C \st 0 < \Re z < 1\}$, and introduce the vector space
\begin{align*}
F(\Bcal_0, \Bcal_1)= \{& f : S \to \Bcal_0 + \Bcal_1, \mbox{ analytic,
extending continuously to } \overline S ,\\
\nonumber & \mbox{  with }\sup_{z \in \overline S}
\norm{f(z)}{\Bcal_0+\Bcal_1} < \infty, \nonumber t \mapsto f(it) \in \Ccal_b (\R,\Bcal_0), 
\\& \mbox{ and } t \mapsto f(1+it) \in \Ccal_b (\R,\Bcal_1)\}.
\end{align*}

Then, we define the following norm on this space:

\begin{equation*}
\norm{f}{F(\Bcal_0,\Bcal_1)}:=\max(\norm{\norm{f(it)}{\Bcal_0}}{\infty}, \norm{\norm{f(1+it)}{\Bcal_1}}{\infty})
\end{equation*}

The complex interpolation space is defined for $\theta \in (0,1)$ by
\begin{equation*}
[\Bcal_0, \Bcal_1]_\theta 
= \{ u \in \Bcal_0+\Bcal_1 \st \exists f \in F(\Bcal_0, \Bcal_1) \mbox{ with } f(\theta)=u\},
\end{equation*}
normed by
\begin{equation*}
\norm{u}{[\Bcal_0, \Bcal_1]_\theta}
= \inf_{f(\theta)=u} \norm{f} {F(\Bcal_0, \Bcal_1)}.
\end{equation*}

The main idea is that the $[\Bcal_0,\Bcal_1]_\theta$ spaces are 
intermediates between $\Bcal_0$ and $\Bcal_1$, ``close to 
$\Bcal_0$'' for low parameters $\theta$ and ``close to $\Bcal_1$'' 
for high parameters. Usually, the embedding can 
be done for example in $\Scal'$, the dual space of the space of 
$\Ccal^\infty$ rapidly decaying functions. We mention here the 
following key result (Theorem 1.(c) of §2.4.2 in \cite{Triebel1}):

\begin{prop}[Interpolation of Sobolev spaces]\quad
\label{Interpolation}

For any $t_0$, $t_1 \in \R_+$, $p_0$, $p_1 \in
(1,\infty)$ and $\theta \in (0,1)$, the interpolation space
$[H_{p_0}^{t_0}, H_{p_1}^{t_1}]_\theta$ is equal to $H_p^t$ for
$t=t_0 (1-\theta) + t_1 \theta$ and
$1/p=(1-\theta)/p_0+\theta/p_1$.

In particular, for all $t \in (0,1)$ and $p \in (1, \infty)$, 
$H_p^t = [\Lbb^p, H_p^1]_t$.
\end{prop}

\begin{rem}[Alternative definition of Sobolev spaces]\quad

By Proposition~\ref{Interpolation}, we have another way to define the $\Hcal_p^t$ spaces on a 
(compact) Riemannian manifold $X$, this time intrinsically. Since 
there exists a Lebesgue measure $\LL$ on $X$, the spaces $\Lbb^p$ 
are well defined. Moreover, the space $\Hcal_p^1$ is none other than 
$\{ u \in \Lbb^p : \nabla u \in \Lbb^p \}$ with the norm $\norm{u}{\Hcal_p^1} 
= \norm{u}{\Lbb^p} + \norm{\nabla u}{\Lbb^p}$, the derivative 
being taken in a weak sense, and defined via the Levi-Civita connection 
(see Theorem 1.(b) of §2.3.3 in \cite{Triebel1}). It is not difficult to see that 
$\Hcal_p^t=[\Lbb^p (X_0),\Hcal_p^1 (X_0)]_t$, with a norm equivalent to the 
one previously defined (Remark~2 p.321 in \cite{Triebel2}).
\end{rem}

Here, the main interest of interpolation theory will be to derive 
inequalities for operators in the $H_p^t$ spaces from inequalities 
in the $\Lbb^p$ and $H_p^1$ spaces, where the manipulations are much 
easier. This will be made possible thanks to the following property (see e.g. 
§1.9 in \cite{Triebel1}): for any interpolation couple
$(\Bcal_0, \Bcal_1)$, $\theta \in (0,1)$ and $\Lcal \in L(\Bcal_0, \Bcal_1)$, we have

\begin{equation}
\label{interpolation_inequality}
\norm{\Lcal}{[\Bcal_0, \Bcal_1]_\theta \to [\Bcal_0,\Bcal_1]_\theta} 
\leq \norm{\Lcal}{\Bcal_0 \to \Bcal_0}^{1-\theta} \norm{\Lcal}{\Bcal_1 \to \Bcal_1}^{\theta}
\end{equation}

\begin{rem}[Sobolev spaces with negative parameter]\quad

The dual spaces of the Sobolev spaces are well known (see e.g. §4.8.1 in \cite{Triebel1}): 
for any $p \in (1,\infty)$ and $t \in \R_+$, $\Hcal_p^t$ 
is reflexive and $(\Hcal_p^t)^*=\Hcal_{p^*}^{-t}$, where 
$1/p+1/p^* = 1$ and the $\Hcal_{p^*}^{-t}$ is defined via 
the Fourier transform the same way as the $\Hcal_p^t$ spaces.
Such Sobolev spaces with negative parameter will appear (although briefly), 
for instance in the proof of Lemma~\ref{lem:sum}.
\end{rem}


\subsection{Functions with bounded variation}

In order to define the space of functions with bounded variation, we proceed in the 
same way as in Subsection~\ref{subsec:SobolevAndFourier}: we first define this space 
on open sets of $\R^d$, and then use the charts to define such a space on a Riemannian manifold.

\begin{defin}[Functions with bounded variation]\quad
\label{def:BV}

Let us denote the Lebesgue measure on $\R^d$ by $\LL$. Let $\Omega$ be an open set of $\R^d$. For 
$u \in \Lbb^1 (\Omega)$, we define the variation of $u$, $V(u)$, by
\[
V(u) = d^{-1} \sup_{\substack{\varphi \in \Ccal_c^\infty (\Omega,\R^d) \\ \norm{\varphi}{\infty} \leq 1}} \int_{\Omega}
 u.div(\varphi) \dL.
\]
The space of functions with bounded variation on $\Omega$, denoted $\BV(\Omega)$, 
is the subspace of functions $u$ of $\Lbb^1 (\Omega)$ such that $V(u) < +\infty$, endowed with the 
norm $\| u \|_{\BV (\Omega)} = \| u \|_{\Lbb^1} + V(u)$.
\end{defin}

\begin{rem}\quad

Formally, we have (see e.g. Remark~1.8 in \cite{Giusti}):
\[
V(u) = d^{-1}\int_{\Omega} | \nabla (u) |.
\]
The derivative being taken in a weak sense, $| \nabla (u) |$ may be a measure with no density 
with respect to the Lebesgue measure. For $\Ccal^1$ functions, it can be taken literally, with
\[
V(u) = d^{-1}\int_{\Omega} \sum_{i=1}^d \left| \frac{\partial u}{ \partial x_i} \right| \dL.
\]
\end{rem}

As for the Sobolev spaces, we may define the space of functions with bounded variation on a compact set $X_0$ of 
a Riemannian manifold $X$ starting from its definition on $\R^d$ (this space also behaves 
well under composition by $\Ccal^{1+\Lip}$ diffeomorphisms, as is shown by Lemma~\ref{lem:BVcomposition_diffeo}) 
and transporting by the charts.

\begin{defin}[Functions with bounded variation on $X_0$]\quad
\label{def:BVspaceX}

Fix  a finite number of $C^{1+\Lip}$
charts $\kappa_1,\dots,\kappa_J$ whose domains of definition
cover a compact neighborhood of $X_0$, and a partition of unity
$\rho_1,\dots,\rho_J$, such that the support of $\rho_j$ is
compactly contained in the domain of definition of $\kappa_j$,
and $\sum \rho_j=1$ on $X$. The space $\BV$ of functions with bounded 
variation on $X_0$ (sometimes also denoted $\BV (X_0)$) is the space 
of functions $u$ supported on $X_0$ such that $(\rho_j
u)\circ \kappa_j^{-1}$ belongs to $\BV (\R^d)$ for all $j$,
endowed with the norm
\begin{equation}
\label{BVNormDefinition}
\norm{u}{\BV}
= \sum_{j=1}^J \norm{ (\rho_j u) \circ \kappa_j^{-1}}{\BV (\R^d)}.
\end{equation}
\end{defin}

Since we can neither use the Fourier transform nor the interpolation theory here, 
we will need different tools for the part of our proof relying on functions with bounded variation. 
We choose to use two theorems dealing with the approximation of functions with bounded 
variation by $\Ccal^\infty$ functions, and then to work with smooth functions 
(for which the computations can be done explicitly). They are 
respectively a variation on Theorem~1.17 and Theorem~1.9 in \cite{Giusti}. The first is 
obtained via convolution with smooth kernels (say, gaussian or $\Ccal^\infty$ with compact 
support), the second follows directly from the definition of the variation.

\begin{theor}[Approximation of functions with bounded variation]\quad
\label{theor:approx1}

For any function $u$ in $\BV (\R^d)$, there exists a sequence of $\Ccal^\infty$ functions 
$(u_n)_{n \in \N}$ which converges in $\Lbb^1$ to $u$, and such that $(\norm{u_n}{\BV (\R^d)})_{n \in \N}$ converges to $\norm{u}{\BV (\R^d)}$.
\end{theor}

\begin{theor}[Control of the approximation]\quad
\label{theor:approx2}

Let $u \in \BV (\R^d)$, and $(u_n)_{n \in \N}$ a sequence of functions with bounded 
variation which converges to $u$ in $\Lbb^1$. Then $\displaystyle \norm{u}{\BV (\R^d)} \leq \liminf_{n \rightarrow +\infty} 
\norm{u_n}{\BV (\R^d)} $.
\end{theor}

\section{Elementary inequalities}
\label{sec:inequalities}

Our goal in this section is to prove several continuity results for operations 
in the $H_p^t$ spaces, for instance multiplication by $\Ccal^\alpha$ functions 
or by characteristic functions of intervals. They show the convenient properties 
of the spaces we are working with, provided we take good values of $t$ and $p$, 
and are virtually sufficient to get the main theorems.

Hence, obtaining versions of the main theorem for other parameters or function 
spaces can be reduced to the obtention of those results of continuity; see for 
instance Section~\ref{sec:variation} for the corresponding results in the $\BV$ space.


\subsection{First inequalities}

The first inequalities are continuity results for some 
linear operations; they were proven in the 60s and 70s.

We first deal with the multiplication by smooth enough functions, 
which will be necessary primarily when we multiply $u$ and $g$, 
and also when it comes to study what happens 
at a small scale (we will multiply the functions $u$ with 
smooth functions with small supports). A proof of the following lemma 
can be found in §4.2.2 of \cite{Triebel2}.

\begin{lem}[Multiplication by $\Ccal^\alpha$ functions]\quad
\label{lem:multiplication_smooth_functions}

Let $0<t<\alpha$ be real numbers. There exists $C_{t,\alpha}$ 
such that, for all $p \in (1,+\infty)$, for all $u \in H_p^t(\R^d)$ 
and $g \in \Ccal^\alpha (\R^d)$,
\begin{equation*}
 \norm{gu}{H_p^t} 
\leq C_{t,\alpha} \norm{g}{\Ccal^\alpha} \norm{u}{H_p^t}.
\end{equation*}
\end{lem}

The next inequality (Corollary I.4.2 from \cite{Strichartz}) 
is central for our study, since it tells us that $H_p^t (\R^d)$ behaves 
well in a ``piecewise setting'' up to constraints on the parameters 
$t$ and $p$. We note that one has to combine 
Corollary I.4.2 from \cite{Strichartz} with its Corollary I.3.7 
to show that the constant $C_{t,p}$ in the following lemma does not depend on $\Omega$
otherwise than via $L$.

\begin{lem}[Multiplication by characteristic functions of nice sets]\quad
\label{lem:multiplication_charfunction}

Let $0<t<1/p<1$ be real numbers; let $L \geq 1$. There exists $C_{t,p}$ 
such that for each measurable subset $\Omega$ of $\R^d$ whose 
intersection with almost every line parallel to some 
coordinate axis has at most $L$ connected components, 
for all $u \in H_p^t(\R^d)$,
\[
 \norm{1_\Omega u}{H_p^t} \leq C_{t,p} L \norm{u}{H_p^t}.
\]
\end{lem}


\subsection{Composition with $\Ccal^{1+\alpha}$ diffeomorphisms}

The injection from $\Hcal_p^t$ into $\Lbb^p$ is compact, $t$ being positive and 
$X_0$ compact (in the case of compact subset of some $\R^d$, this is for instance Lemma~2.2 in \cite{Baladi2}; 
the corresponding result for compact subset of Riemannian manifolds follows immediatly). 
Thus, Lemma~\ref{lem:composition_diffeo} (a version of Lemma~24 in \cite{BaladiHyp1}) is 
essential in the way it gives the decomposition 
of $\| \Lcal_g u \|_{\Hcal_p^t}$ into a part bounded by 
$\| u \|_{\Hcal_p^t}$, and a part bounded by $\| u \|_{\Lbb^p}$. 
By Hennion's theorem \cite{Hennion}, the part bounded by $\| u \|_{\Lbb^p}$ 
will not affect our estimation of the essential spectral radius of $\Lcal_g$, 
so that only the part bounded by $\| u \|_{\Hcal_p^t}$ matters.

When we apply Lemma~\ref{lem:composition_diffeo}, $F$ will be the local inverse of an iterate of $T$, 
i.e. some $T_\ii^{-n}$ (or rather the maps on some open set of $\R^d$ obtained 
by transporting $T_\ii^{-n}$ via the charts), and $A$ will be a local approximation 
of $DT_\ii^{-n}$. If $A$ is such an approximation in a neighborhood of some 
$x$, then, $T$ being expanding, we can write $\|A\| \leq \tilde{\lambda_n}^{-1}$ where $\tilde{\lambda_n}$ is 
the essential infimum of $\lambda_n$ on this neighborhood of $x$.

\begin{lem}[Composition with smooth diffeomorphisms]\quad
\label{lem:composition_diffeo}

Let $F \in \Ccal^1 (\R^d,\R^d)$ be a diffeomorphism and $A \in GL_d (\R)$ such that, for all 
$z \in \R^d$, $\| A^{-1} \circ DF(z) \| \leq 2$ and $\| DF(z)^{-1} \circ A \| \leq 2$.

Then, for all $t \in [0,1]$, for all $p \in (1,+\infty)$, there exist constants $C_{t,p}$ 
and $C'_{t,p}$, which do not depend on $F$ nor on $A$, such that, for all $u \in H_p^t$,

\begin{equation*}
\norm{u \circ F}{H_p^t} 
\leq C_{t,p}  | \det A |^{-\frac{1}{p}} \|A\|^t \norm{u}{H_p^t} + C'_{t,p} |\det A|^{-\frac{1}{p}} \norm{u}{\Lbb^p}
\end{equation*}
\end{lem}

\begin{proof}\quad

We will write $u \circ F = u \circ A \circ A^{-1} \circ F$, and put $\tilde{F} = A^{-1} \circ F$. First we deal with $u \circ A$, and then with $u \circ \tilde{F}$.

\begin{gras}First step:\end{gras} $u \circ A$

We want to estimate $\norm{u \circ A}{H_p^t} = \norm{\Fcal^{-1} 
( a_t \Fcal (u \circ A)) }{\Lbb^p}$, where $a_t(\xi) = (1+| \xi |^2)^{t/2}$. 
A change of variables gives $\Fcal^{-1} ( a_t \Fcal (u \circ A)) 
= \Fcal^{-1} ( a_t \circ \transposee{A} \Fcal (u)) \circ A$.

If $|\xi| > 1/\|A\|$, we have

\[
(a_t \circ \transposee{A})(\xi) 
= (1+| \transposee{A} \xi |^2)^{t/2} 
\leq (1+ \|A\|^2 | \xi |^2)^{t/2} 
\leq 2^{t/2} \|A\|^t a_t (\xi).
\]

On the other hand, if $|\xi| \leq 1/\|A\|$, we have

\[
(a_t \circ \transposee{A}) (\xi) 
\leq (1+ \|A\|^2 | \xi |^2)^{t/2} 
\leq 2^{t/2}.
\]

Finally, we get $a_t \circ \transposee{A} \leq 2^{t/2} \|A\|^t a_t + 2^{t/2}$, 
and by the same means the same kind of upper bound for the 
$\displaystyle \xi^\gamma D^\gamma \left( \frac{a_t \circ \transposee{A}}{\|A\|^t a_t + 1} \right)$ 
for all $\gamma \in \{0,1\}^d$. By Theorem~\ref{thm:Marcinkiewicz}, there exist $ C_{t,p}$ 
and $C'_{t,p}$ such that

\begin{align*}
\norm{\Fcal^{-1} ( a_t \circ \transposee{A} \Fcal (u))}{\Lbb^p} &
\leq C_{t,p} |\det A|^{-1/p} \|A\|^t \norm{\Fcal^{-1} ( a_t \Fcal (u))}{\Lbb^p} \\&
+ C'_{t,p} |\det A|^{-1/p} \norm{\Fcal^{-1} ( \Fcal (u))}{\Lbb^p} \\&
\leq C_{t,p} |\det A|^{-1/p} \|A\|^t \norm{u}{H_p^t} + C'_{t,p} |\det A|^{-1/p} \norm{u}{\Lbb^p}.
\end{align*}

The following estimate ensues:

\begin{equation}
\label{eq:premiere_inegalite_diffeogen}
\norm{u \circ A}{H_p^t} 
\leq C_{t,p} |\det A|^{-1/p} \|A\|^t \norm{u}{H_p^t} + C'_{t,p} |\det A|^{-1/p} \norm{u}{\Lbb^p}.
\end{equation}

This inequality ends the first part of the proof.

\smallskip

\begin{gras}Second step:\end{gras} $u \circ \tilde{F}$

We will use interpolation. Notice that, since $\| D\tilde{F} \| 
\geq 1/2$ everywhere, $| \det \tilde{F}^{-1} | \geq 2^{d}$ and 
$\norm{u \circ \tilde{F}}{\Lbb^p} \leq 2^{d/p} \norm{u}{\Lbb^p}$ for 
all $u \in \Lbb^p$. We also get the same way $\norm{u \circ \tilde{F}}{H_p^1} 
\leq C_p 2^{d/p+1} \norm{u}{H_p^1}$ for 
some $C_p$, since the $H_p^1$ norm is equivalent to 
$\norm{u}{\Lbb^p} + \norm{Du}{\Lbb^p}$ (see for instance §2.3.3 in \cite{Triebel1}).

Hence, by Proposition~\ref{Interpolation}, there exists $C_{t,p}$ such that, for all $t \in [0,1]$, for all $u \in H_p^t$,

\begin{equation}
\label{eq:seconde_inegalite_diffeogen}
\norm{u \circ \tilde{F}}{H_p^t} \leq C_{t,p}  2^{d/p} \norm{u}{H_p^t}.
\end{equation}

The lemma follows immediately from \eqref{eq:premiere_inegalite_diffeogen} and \eqref{eq:seconde_inegalite_diffeogen}.
\end{proof}

Since the injection $H_p^t \to \Lbb^p$ is continuous ($t$ being non-negative), a consequence 
of \eqref{eq:premiere_inegalite_diffeogen} is that, for all 
$p \in (1,+\infty)$ and $t \in [0,1]$, there exists a constant 
$C_{t,p}$ such that, for all $u \in H_p^t(\R^d)$, for all $A \in GL_d (\R)$,
\begin{equation}
\label{eq:composition_matrix}
\| u \circ A \|_{H_p^t} 
\leq C_{t,p} \sup \{\|A\|, 1\} | \det A |^{-1/p} \| u \|_{H_p^t}.
\end{equation}


\subsection{Localization}

For the proof of Theorem~\ref{theor:spectral_gap}, we work locally, on small open sets, and the 
compactness of $X_0$ allows us to work globally with a 
finite number of such sets. However, we need to control the 
intersection multiplicity of those sets. This control will be 
obtained with a partition of the unity and a "zoom" to get 
some regularity when working at small scales.

\begin{lem}[Localization principle]\quad
\label{lem:localization}

Let $\eta \in \Ccal_c^\infty (\R^d,\R)$, and write, for all 
$x \in \R^d$ and $m \in \Z^d$, $\eta_m (x) = \eta (x+m)$. Then, 
there exists a constant $C_{t,p,\eta}$ such that, for all $u \in H_p^t$,
\begin{equation}
\left( \sum_{m \in \Z^d} \norm{\eta_m u}{H_p^t}^p \right)^\frac{1}{p} 
\leq C_{t,p,\eta} \norm{u}{H_p^t}.
\end{equation}
\end{lem}

This lemma is Theorem 2.4.7 from \cite{Triebel2}.

The constant $C_{t,p,\eta}$ depends on $t$, $p$, on the support 
of $\eta$ and on its $\Ccal^k$ norm for some large enough $k$. 
The proof we give of a localization principle in Section~\ref{sec:variation} will make explicit 
a kind of dependence of $C_{\BV,\eta}$ (the corresponding bound for the space of 
functions with bounded variation) in $\eta$.


\begin{lem}\quad
\label{lem:sum}

Let $1<p<+\infty$ and $t \in \R_+$. There exists a constant $C_{t,p}$ such that, for
any functions $v_1,\dots, v_l$ with compact support in
$\R^d$, belonging to $H_p^t$, there exists a constant $C_{t,p,\{v\}}$,
depending only on the supports of the functions $v_i$, with

\begin{equation*}
\norm{ \sum_{i=1}^l v_i}{H_p^t}^p
\leq C_{t,p} m^{p-1} \sum_{i=1}^l \norm{v_i}{H_p^t}^p 
+ C_{t,p,\{v\}} \sum_{i=1}^l \norm{v_i}{\Lbb^p}^p,
\end{equation*}

where $m$ is the intersection multiplicity of the supports of
the $v_i$'s, i.e. $m = \sup_{x \in \R^d} \Card\{i \st x \in
\Supp(v_i)\}$.
\end{lem}

\begin{proof}\quad

Let $B$ be the operator acting on functions by $B
v = \Fcal^{-1}((1+|\xi|^2)^{t/2} \Fcal
v)$, so that $\norm{v}{H_p^t}=\norm{Bv}{\Lbb^p}$.

By Lemma~2.7 in \cite{Baladi2}, for any compact $K$ and any
neighborhood $K'$ of $K$, there exist $C_{K,K',t,p} > 0$ and a
function $\Psi:\R^d\to [0,1]$ equal to $1$ on $K$ and vanishing
on the complement of $K'$, such that for any $v \in H_p^t$ 
with support in $K$,
\begin{equation}
\norm{\Psi Bv - Bv}{\Lbb_p} \leq C_{K,K',t,p} \norm{v}{H_p^{t-1}}.
\end{equation}

Let $v_1,\dots,v_l$ be functions with compact supports
whose intersection multiplicity is $m$. Choose neighborhoods
$K'_1,\dots,K'_l$ of the supports of the $v_i$s
whose intersection multiplicity is also $m$, and functions
$\Psi_1,\dots,\Psi_l$ as above. Then
\begin{equation}
\label{previous}
\norm{\sum_i v_i}{H_p^t}^p = \norm{\sum_i Bv_i}{\Lbb^p}^p
\leq \norm{\sum_i \Psi_i Bv_i}{\Lbb^p}^p + C_{K,K',t,p} \sum_i \norm{v_i}{H_p^{t-1}}^p.
\end{equation}
Since $t \leq 1$, the inclusion $\Lbb^p \to H_p^{t-1}$ is continuous, 
and $\norm{v_i}{H_p^{t-1}}^p \leq C_p \norm{v_i}{\Lbb^p}^p$.

By convexity, the inequality $(x_1+\dots+x_m)^p \leq
m^{p-1}\sum x_i^p$ holds for any nonnegative numbers
$x_1,\dots,x_m$. Since the multiplicity of the $K'_i$ is at
most $m$, this yields
\[
\left|\sum_i \Psi_i Bv_i \right|^p \leq m^{p-1} \sum_i |Bv_i|^p.
\]
Integrating this inequality and using (\ref{previous}), we
get the lemma.
\end{proof}


\section{Proof of the main theorem}
\label{sec:main}

Before proving Theorem~\ref{theor:spectral_gap}, we need to show that 
$T$ does not "hack" too much the functions, or in other 
words that its discontinuities are not so bad that functions 
in $\Hcal_p^t$ do not stay in this space when composed by $T$. 
The following lemma will allow us to apply Lemma~\ref{lem:multiplication_charfunction}. 
For once, its proof is adapted from \cite{BaladiHyp2} instead of \cite{BaladiHyp1}.

\begin{lem}\quad
\label{lem:truncate}

For $1 \leq i \leq I$, let $L_i$ be the number of smooth boundary components
of $O_i$, and $L = \max_i L_i$.

For any $n \geq 1$, for any $\ii=(i_0,\dots,i_{n-1})$, for any $x\in
\overline{O_\ii}$, for any $j$ such that $x\in \Supp \rho_j$,
there exists a neighborhood $O'$ of $x$ and an orthogonal 
matrix $M$ such that the intersection of $M \kappa_j(O'\cap O_\ii)$ 
with almost any line parallel to a coordinate axis has at most $Ln$ components.
\end{lem}

\begin{proof}\quad

If $x$ is in the interior of some $O_\ii$, we may just take for $O'$ a ball 
small enough so that $O' \subset O_\ii$, and $M=Id$. Let assume that $x$ belongs to 
a backwards image of a $\Ccal^1$ hypersurface of some $\partial O_{i_k}$ by $T_\ii^k$, with $k \leq n$.






On a small neighbourhood of $x$, the smooth boundary components of $O_\ii$ 
are close to their tangent hyperplanes in $x$. We can choose $d$ 
orthogonal vectors of norm $1$ all transverse to all those hyperplanes. The 
smooth boundary components are $\Ccal^1$, so that the vectors we chose are 
transverse to the boundary components (an not only to the hyperplanes) 
on a small enough ball around $x$, that we shall denote $O'$. 

We have $\displaystyle O_\ii = \bigcap_{k=0}^{n-1} T^{-k} O_{i_k}$, 
each of the $T^{-k} O_{i_k}$ being an open set bounded by at most 
$L$ hypersurfaces; hence, $O_\ii$ is bounded by at most 
$Ln$ preimages of those hypersurfaces under some $T^{-k}$. By construction, each line parallel 
to any coordinate axis in the new basis intersects any of these backward images of 
hypersurfaces in at most one point, so that their intersection with 
$O_\ii \cap O'$ has at most $Ln$ connected components. All we need is to take for $M$ 
the change-of-basis natrix from the new orthogonal basis to the canonical one.
\end{proof}

We now have all the tools we need to prove the main theorem.

\begin{proof}[Proof of Theorem~\ref{theor:spectral_gap}]\quad

Let $p$ and $t$ be as in the assumptions of the theorem.
Let $n>0$, and $r_n>1$ (the precise value of $r_n$ will be
chosen later). We define a dilation $R_n$ on $\R^d$ by $R_n(z)=r_n z$.
Let $\norm{u}{n}$ be another norm on $\Hcal_p^t$,
given by

\begin{equation}
\label{eq:zoom}
\norm{u}{n}=\sum_{j=1}^J \norm{(\rho_j u)\circ \kappa_j^{-1} \circ R_n^{-1}}{H_p^t}.
\end{equation}

The norm $\norm{u}{n}$ is of course equivalent to the usual norm on
$\Hcal_p^t$, but  we
look at the space $X_0$ at a smaller scale. Functions are
much flatter at this new scale, so that we have no problem 
using Lemma~\ref{lem:multiplication_smooth_functions}, 
which involves their $\Ccal^\alpha$ norm: we will replace it by 
their $\Lbb^\infty$ norm. This will also enable us to use partitions of
unity $(\eta_m)$ with very small supports without spoiling the estimates.
The use of this "zooming" norm is similar to the good choice of
$\epsilon_0$ in \cite{Saussol}.

We will prove that there exists $C_{t,p,\eta,\alpha}$ such that, 
if $n$ is fixed and $r_n$ is large enough, then there exists $C_{t,p,n}$ such that:

\begin{equation}
\label{eq:a_prouver}
\norm{\Lcal_g^n u}{n}^p 
\leq C_{t,p,n} \norm{u}{\Lbb^p}^p + C_{t,p,\eta,\alpha}D_n^b (D_n^e)^{p-1} (Ln)^p \norm{ |\det DT^n| \lambda_n^{-tp} |g^{(n)}|^p}{\Lbb^\infty} \norm{u}{n}^p.
\end{equation}

We have already observed that, since $X_0$ is compact and $t>0$, the injection of $\Hcal_p^t$ into $\Lbb^p$ is
compact (see e.g. Lemma~2.2 in \cite{Baladi2}). Hence, by Hennion's theorem (actually Corollary~1 in \cite{Hennion}), the essential
spectral radius of $\Lcal_g^n$ acting on $\Hcal_p^t$ (for
either $\norm{u}{\Hcal_p^t}$ or $\norm{u}{n}$, since these
norms are equivalent) is at most

\begin{equation}
\Bigl[C_{t,p,\eta,\alpha} (Ln)^p D_n^b (D_n^e)^{p-1} \norm{ |\det DT^n| \lambda_n^{-tp} |g^{(n)}|^p}{\Lbb^\infty}\Bigr]^\frac{1}{p}.
\end{equation}

Taking the power $1/n$ and letting $n$ tend to $+ \infty$, we
obtain Theorem~\ref{theor:spectral_gap} since the quantity
$(C_{t,p,\eta,\alpha}^{1/p}(Ln)^p)^{1/n}$ converges to $1$.

\smallskip

It remains to prove \eqref{eq:a_prouver}, for large enough $r_n$.
The estimate will be subdivided into three steps:

\begin{enumerate}
\item Decomposing $u$ into a sum of functions $v_{j,m}$ with small supports and
well controlled $\norm{\cdot}{n}$ norms.
\item Estimating each term $(1_{O_\ii} g^{(n)} v_{j,m})\circ
T_\ii^{-n}$, for $i$ of length $n$.
\item Adding all  terms to obtain $\Lcal_g^n u$.
\end{enumerate}

\smallskip

\begin{gras}First step:\end{gras}

For $1\leq j \leq J$ and $m \in \Z^d$, let
$\tilde v_{j,m} = \eta_m \cdot(\rho_j u)\circ \kappa_j^{-1}\circ
R_n^{-1}$, where $\eta_m(x)=\eta(x+m)$, with $\eta: \R^d \to [0,1]$  a
compactly supported $\Ccal^\infty$ function so that
$\displaystyle \sum_{m \in \Z^d} \eta_m=1$.
Since the intersection multiplicity of the supports of the functions $\eta_m$
is bounded, this is also the case for the $\tilde v_{j,m}$.
Moreover, if $j$ is fixed, we get, using Lemma~\ref{lem:localization},

\begin{align}
\label{eq:decompose}
\sum_{m \in \Z^d} \norm{\tilde v_{j,m}}{H_p^t}^p &
= \sum_{m \in \Z^d} \norm{\eta_m \cdot (\rho_j u)\circ \kappa_j^{-1}\circ R_n^{-1}}{H_p^t}^p \\ \nonumber &
\leq C_{t,p,\eta} \norm{(\rho_j u)\circ \kappa_j^{-1}\circ R_n^{-1}}{H_p^t}^p
\leq C_{t,p,\eta} \norm{u}{n}^p.
\end{align}

Since $R_n$ expands the distances by a factor $r_n$ while the
size of the supports of the functions $\eta_m$ is uniformly
bounded, the supports of the functions
$$v_{j,m}= \tilde v_{j,m}\circ R_n \circ \kappa_j=\eta_m \circ R_n \circ \kappa_j\cdot(\rho_j u)
$$
are arbitrarily small if $r_n$ is large enough.
Finally

\begin{equation}
u=\sum_j \rho_j u = \sum_{j,m}   v_{j,m}.
\end{equation}

\smallskip

\begin{gras}Second step:\end{gras}

Fix $j,k\in \{1,\dots,J\}$, $m\in \Z^d$ and
$\ii=(i_0,\dots,i_{n-1})$.
We will prove that

\begin{equation}
\label{eq:a_prouver_2}
\norm{ (\rho_k (g^{(n)}1_{O_\ii} v_{j,m})\circ T_\ii^{-n})\circ \kappa_k^{-1}\circ R_n^{-1}}{H_p^t}
\end{equation}
\[
\leq C_{t,p,n} \norm{u}{\Lbb^p}+ C_{t,p} (Ln) \norm{ |\det DT^n|^\frac{1}{p} g^{(n)} \lambda_n^{-t}}{\Lbb^\infty}\norm{\tilde v_{j,m}}{H_p^t}.
\]

First, if the support of $v_{j,m}$ is small enough (which can
be ensured by taking $r_n$ large enough), there exists a
neighborhood $O$ of this support and a matrix $M$ satisfying
the conclusion of Lemma~\ref{lem:truncate} for all $x$ in $O_\ii$. Therefore, the
intersection of $R_n (M (\kappa_j(O\cap O_\ii)))$ with almost
any line parallel to a coordinate axis has at most $Ln$ connected components. 
Hence, Lemma~\ref{lem:multiplication_charfunction} implies
that the multiplication by $1_{O\cap O_\ii}\circ \kappa_j^{-1}
\circ M^{-1}\circ R_n^{-1}$ sends $H_p^t$ into itself,
with a norm bounded by $C_{t,p} Ln$. Using the fact that $M$ and
$R_n$ commute, the properties of $M$, and \eqref{eq:composition_matrix}, we get

\begin{equation}
\label{eq:v'}
\norm{ 1_{O_\ii}\circ \kappa_j^{-1}\circ R^{-1}_n \cdot \tilde v_{j,m}}{H_p^t}
\leq C_{t,p} Ln \norm{\tilde v_{j,m}}{H_p^t}.
\end{equation}

Next, let $$\tilde v_{j,k,m}=((\rho_k \circ T_\ii^n) 1_{O_\ii})\circ
\kappa_j^{-1}\circ R^{-1}_n \cdot \tilde v_{j,m}$$ (we suppress
$i$ from the notation for simplicity). Let  also $\chi$ be a
$\Ccal^\infty$ function supported in the neighborhood $O$ of the
support of $v_{j,m}$ with $\chi\equiv 1$ on this support. Up to
taking larger $r_n$ we may ensure that  $\norm{(\chi (\rho_k
\circ T_\ii^n ))\circ \kappa_j^{-1}\circ R_n^{-1}}{C^\alpha} \leq
2$. Then Lemma~\ref{lem:multiplication_smooth_functions} and \eqref{eq:v'} imply

\begin{equation}
\label{eq:v_sec}
\norm{\tilde v_{j,k,m}}{H_p^t}
\leq C_{t,p,\alpha} Ln \norm{\tilde v_{j,m}}{H_p^t}.
\end{equation}

In addition, putting $G = \kappa_j \circ T_\ii^{-n} \circ \kappa_k^{-1}$, we have

\begin{align}
((\rho_k \circ T_\ii^n)1_{O_\ii} v_{j,m})\circ T_\ii^{-n} \circ \kappa_k^{-1}\circ R_n^{-1} &
= \tilde v_{j,k,m} \circ R_n\circ \kappa_j \circ T_\ii^{-n} \circ \kappa_k^{-1}\circ R_n^{-1} \\ \nonumber &
= \tilde v_{j,k,m} \circ R_n \circ G \circ R_n^{-1}.
\end{align}

Applying Lemma~\ref{lem:composition_diffeo} to 
$F=R_n \circ G \circ R_n^{-1}$, we get (for some point
$x$ in the support of $v_{j,m}$, and some matrix $A$ of the
form $DG(R_n^{-1}(x))$ for some $x$, so that $\|A\| \leq \lambda_n^{-1} (x)$)

\begin{align}
\label{eq:glu}
\norm{\tilde v_{j,k,m} \circ R_n \circ G \circ R_n^{-1}}{H_p^t} &
\leq C_{t,p,n} \norm{u}{\Lbb^p} \\ \nonumber &
+C_{t,p}  |\det A|^{-\frac{1}{p}}
\lambda_n(x)^{-t} \norm{\tilde v_{j,k,m}}{\Hcal_p^t}.
\end{align}

The constant $C_{t,p,n}$ also depends on the choice of $A$. 
Let $\chi$ be a $\Ccal^\infty$ function supported in $O'$ with
$\chi \equiv 1$ on the support of $v_{j,m}\circ T_\ii^{-n}$.
For $\delta>0$, we can ensure by increasing $r_n$ that  the
$\Ccal^\alpha$ norm of $(\chi g^{(n)} )\circ T_\ii^{-n}
\circ\kappa_k^{-1}\circ R_n^{-1}$ is bounded by
$|g^{(n)}(x)|+\delta$ for some $x$ in the support of $v_{j,m}$.
Choosing $\delta>0$ small enough, we deduce from \eqref{eq:glu},
Lemma~\ref{lem:multiplication_smooth_functions} and \eqref{eq:v_sec}

\[
\norm{ (\rho_k(g^{(n)}1_{O_\ii} v_{j,m})\circ T_\ii^{-n}) \circ \kappa_k^{-1}\circ R_n^{-1}}{H_p^t}
\]
\[
\leq C_{t,p,n} \norm{u}{\Lbb^p}+ C_{t,p,\alpha} Ln \norm{ |\det DT^n|^\frac{1}{p} g^{(n)}\lambda_n^{-t}}{\Lbb_\infty}\norm{\tilde v_{j,m}}{H_p^t}.
\]

This proves \eqref{eq:a_prouver_2}.

\smallskip

\begin{gras}Third step:\end{gras}

We have $\displaystyle \Lcal_g^n u =\sum_{j,m} \sum_\ii
(1_{O_\ii} g^{(n)} v_{j,m})\circ T_\ii^{-n}$. (Note that only
finitely many terms in this sum are nonzero by compactness of
the support of each $\rho_j$.) We claim that the intersection
multiplicity of the supports of the functions $(1_{O_\ii}
g^{(n)} v_{j,m})\circ T_\ii^{-n}$ is bounded by $C_{\eta} D_n^e$.
Indeed, this follows from the fact that any point $x\in X_0$
belongs to at most $D_n^e$ sets $\overline{T_\ii^n(O_\ii)}$, and
that the intersection multiplicity of the supports of the
functions $v_{j,m}$ is bounded.

To estimate $\norm{\Lcal_g^n u}{n}$, we have to bound each
term $\norm{ (\rho_k \Lcal_g^n u)\circ \kappa_k^{-1}\circ
R_n^{-1}}{H_p^t}$, for $1\leq k\leq J$. Let us fix such a
$k$. By Lemma~\ref{lem:sum}, we have

\[
\norm{ (\rho_k \Lcal_g^n u)\circ \kappa_k^{-1}\circ R_n^{-1}}{H_p^t}^p
\]
\[
\leq C_{t,p,n} \norm{u}{\Lbb^p}^p+ C_{t,p,\alpha} (C_{t,p,\eta} D_n^e)^{p-1}\sum_{j,m,\ii}
\norm{(\rho_k(1_{O_\ii} g^{(n)} v_{j,m})\circ T_\ii^{-n}) \circ \kappa_k^{-1} \circ R_n^{-1}}{H_p^t}^p.
\]

We can bound each term in the sum using \eqref{eq:a_prouver_2}
and the convexity inequality $(a+b)^p \leq 2^{p-1}(a^p+b^p)$.
Moreover, for any $(j,m)$, the number of parameters $i$ for
which the corresponding term is nonzero is bounded by the
number of sets $\overline{O_\ii}$ intersecting the support of
$v_{j,m}$. Choosing $r_n$ large enough, we can ensure that the
supports of the $v_{j,m}$ are small enough so that this number is
bounded by $D_n^b$. Together with \eqref{eq:decompose}, this
concludes the proof of \eqref{eq:a_prouver}, and of Theorem~\ref{theor:spectral_gap}.
\end{proof}


\section{Existence of finitely many physical measures}
\label{sec:physical}

In this section, we prove Theorem~\ref{thm:ExistSRB}. In view of Theorem~\ref{theor:BVspectral_gap}, 
and with a few minor adaptations, the same proof leads to Theorem~\ref{thm:BVExistSRB}. 
This proof is a simplified version of Appendix~B in \cite{BaladiHyp1}, which is itself adapted from \cite{Blank}.

\begin{proof}\quad

\begin{gras}First step: the spectrum\end{gras}

First, note that the Lebesgue measure $\LL$ on the manifold 
$X_0$ is in $(\Hcal_p^t)^*$ and is a fixed point for $\Lcal^*$, so that $1$ 
is an eigenvalue of $\Lcal^*$ and thus is in the spectrum of $\Lcal$. 
By Corollary~\ref{cor:spectral_gap} the essential spectral radius $\rho_{ess}$ of 
$\Lcal$ is strictly smaller than $1$, and $1$ is an eigenvalue of $\Lcal$.

We also have the inclusions $\Hcal_p^t \subset \Lbb^p \subset \Lbb^1$, 
so that $\Lbb^\infty \subset \Hcal_{p^*}^{-t}$ and any essentially bounded 
function on $X_0$ can be seen as a linear functional on $\Hcal_p^t$.

In the following, when $\gamma$ denotes an eigenvalue of $\Lcal$ of 
modulus strictly bigger than the essential spectral radius, $E_\gamma$ will denote the 
corresponding finite-dimensional eigenspace and $\pi_\gamma: \Hcal_p^t \to E_\gamma$ 
the corresponding canonical projection. Consider such an eigenvalue 
$\gamma$; since $\Lbb^\infty \cap \Hcal_p^t$ is dense is $\Hcal_p^t$, 
its image by $\pi_\gamma$ in $E_\gamma$ is a dense subspace of $E_\gamma$, and 
therefore is $E_\gamma$ itself.

Let us first prove that $\Lcal$ has no eigenvalue of modulus 
strictly bigger than $1$, nor a nontrivial Jordan block for an eigenvalue 
of modulus $1$. Otherwise, let $\gamma$ be an eigenvalue of 
$\Lcal$ of maximal modulus, with a Jordan block of maximal 
size $d$. Therefore, there exists a bounded function $f$ such 
that $n^{-d}\sum_{i=0}^{n-1}\gamma^{-i} \Lcal^i f$ converges to a
nonzero limit $u$. For any $g \in \Lbb^\infty$,
\begin{equation*}
\langle u, g\rangle 
:= \int_{X_0} ug \dL
= \lim_{n \to +\infty} \frac{1}{n^d} \sum_{i=0}^{n-1} \gamma^{-i} \langle \Lcal^i f, g \rangle
= \lim_{n \to +\infty} \frac{1}{n^d} \sum_{i=0}^{n-1} \gamma^{-i} \int f\cdot g\circ T^i \dL.
\end{equation*}
If $|\gamma|>1$ or $d \geq 2$, this quantity converges to $0$
when $n \to +\infty$ since $\int f \cdot g \circ T^i \dL$ is
uniformly bounded. This contradicts the fact that $u$ is
nonzero.

Let us take any eigenvalue $\gamma$ of modulus $1$. Since 
$\gamma$ is of maximal modulus and $\rho_{ess}<1$, the 
eigenprojection is given by
\begin{equation}
\pi_\gamma f = \lim_{n \to +\infty} \frac{1}{n} \sum_{i=0}^{n-1} \gamma^{-i} \Lcal^i f,
\end{equation}
where the convergence holds in $\Hcal_p^t$.

Let $u \in E_\gamma$. Obviously, $E_\gamma \subset \Hcal_p^t \subset \Lbb^p \subset \Lbb^1$ 
and $u$ can be seen as the density of a finite complex measure with respect to the Lebesgue measure.

For any $f\in \Lbb^\infty \cap \Hcal_p^t$ and any non-negative measurable bounded function $g$,

\begin{align*}
\left| \langle \pi_\gamma f, g \rangle \right| &
\leq \norm{f}{\Lbb^\infty} \lim_{n \to +\infty} \frac{1}{n} \sum_{i=0}^{n-1}  \int \left|g \circ T^i \right| \dL \\&
\leq \norm{f}{\Lbb^\infty} \lim_{n \to +\infty} \left| \frac{1}{n} \sum_{i=0}^{n-1}  \int g \Lcal^i 1 \dL \right| \\&
\leq \norm{f}{\Lbb^\infty} \left| \langle \pi_1 1, g \rangle \right|.
\end{align*}

This means that the measures $u \dL$, where $u$ belongs to some 
$E_\gamma$ with $| \gamma |=1$, are all absolutely
continuous with respect to the reference measure
$\mu = \pi_1 1 \dL$, with bounded density.

\smallskip

For any $u\in E_\gamma$, we write $u \dL =\varphi_u \mu$, where
$\varphi_u \in \Lbb^\infty (\mu)$ is defined $\mu$-almost everywhere.
The equation $\Lcal u=\gamma u$ translates into $T^*(\varphi_u
\mu)=\gamma \varphi_u \mu$. Hence, since $\mu$ is invariant,
\begin{align*}
\int | \varphi_u \circ T -\gamma^{-1} \varphi_u|^2 \dd\mu &
= \int |\varphi_u|^2 \circ T \dd\mu + \int |\gamma^{-1} \varphi_u|^2 \dd\mu -2\Re
\int \overline{\varphi_u} \circ T  \gamma^{-1}\varphi_u \dd\mu \\&
= 2 \int |\varphi_u|^2 \dd\mu -2 \Re\int \gamma^{-1}\overline{\varphi_u}\dd T^*(\varphi_u \mu) 
= 0.
\end{align*}
Let $F_\gamma=\{ \varphi \in \Lbb_\infty(\mu) \st \varphi \circ
T= \gamma^{-1} \varphi\}$. The map $\Phi_\gamma: u \to \varphi_u$
is an injective morphism from $E_\gamma$ to $F_\gamma$. We now 
show that $\Phi_\gamma$ is also surjective, or in other words that 
the functions $\varphi_u \pi_1 1$ always belong to $\Hcal_p^t$.

Let be $\varphi \in F_\gamma$. Since the continuous functions 
are dense in $\Lbb^1 (\mu)$, and any continuous function can be 
uniformly approximated (thus approximated in $\Lbb^1 (\mu)$, all our 
measures being finite) by a sequence of $\Ccal^1$ functions, 
$\Ccal^\alpha \subset \Hcal_p^t$ is dense 
in $\Lbb^1 (\mu)$. We choose an approximation $(\varphi_m)$ of $\varphi$ 
by $\Ccal^\alpha$ functions in $\Lbb^1 (\mu)$, and put $u_m = \pi_\gamma (\varphi_m \pi_1 1)$.

The functions $(u_m)$ belong to $\Hcal_t^p (\LL)$ by Lemma~\ref{lem:multiplication_smooth_functions}. 
Moreover, $|\Lcal u| \leq \Lcal |u|$ for any $u \in \Hcal_t^p$, so that 
$\displaystyle \norm{\Lcal u}{\Lbb^1 (\LL)} \leq \norm{u}{\Lbb^1 (\LL)}$.

For any $f \in \Ccal^\alpha$, we have:
\[
\norm{\pi_\gamma (f \pi_1 1)}{\Lbb^1(\LL)} 
\leq \liminf_{n \to +\infty} \frac{1}{n} \sum_{i=0}^{n-1} \norm{\Lcal^i (f \pi_1 1)}{\Lbb^1 (\LL)}
\leq \norm{f}{\Lbb^1 (\mu)}
\]

Hence, $f \to \pi_\gamma (f \pi_1 1)$ is continuous from $\Ccal^\alpha$ 
(endowed with the semi-norm $\displaystyle \norm{\cdot}{\Lbb^1 (\mu)}$) to $E_\gamma$ 
(endowed with the norm $\displaystyle \norm{\cdot}{\Lbb^1 (\LL)}$). 
As a consequence, $(u_m)$ is a Cauchy sequence in $E_\gamma$, which is of finite dimension 
and thus complete. The sequence $(u_m)$ converges to some $u \in E_\gamma$. 
Then, we have $\Phi_\gamma (u)=\varphi$, and $\Phi_\gamma$ is an isomorphism.

The eigenvalues of $\Lcal$ of modulus $1$ are exactly those
$\gamma$ such that $F_\gamma$ is not reduced to $0$. 
We have $\varphi_\gamma^n \in F_{\gamma^n}$ for all $n$ whenever 
$\varphi_\gamma \in F_\gamma$, so that this set is an union of groups. 
Since $\Lcal$ only has a finite number of eigenvalues of 
modulus $1$, this implies that these eigenvalues are roots 
of unity. In particular, there exists $N>0$ such that 
$\gamma^N=1$ for any eigenvalue $\gamma$.

We now have a good description of the spectrum of $\Lcal$: 
its radius is $1$, its essential radius strictly smaller than $1$, and 
the eigenvalues of modulus $1$ form a finite group, none of 
them having a nontrivial Jordan block. The measures corresponding 
to those eigenvalues are all absolutely continuous with bounded density with respect 
to a reference invariant measure, $\pi_1 1 \dL$, and are all 
absolutely continuous with respect to the Lebesgue measure.

\smallskip
\begin{gras}Second step: mixing physical measures\end{gras}

Let us now assume that $1$ is the only eigenvalue of $\Lcal$ of
modulus $1$ (in the general case, this will be true for
$\Lcal^N$, so we will be able to deduce the general case from
this particular case). Under this assumption,
$\Lcal^n u$ converges to $\pi_1 u$ for any $u \in \Hcal_p^t$.

Consider the subset of $F_1$ (the $T$-invariant functions of $\Lbb(\mu)^\infty$) 
given by the nonnegative functions whose integral with respect to the measure
 $\mu$ is $1$. It is nonempty, since it contains the function $\mu(X_0)^{-1}$.
It is a bounded convex subset of $F_1$, whose extremal points are of the
form $1_B$ for some minimal invariant set $B$ defined $\mu$-almost everywhere. 
Such extremal points are automatically linearly independent; we can also assume 
that they are subsets of $\Supp (\mu)$. Since $F_1$ is
finite-dimensional, there is only a finite number of them, say
$1_{B_1},\dots, 1_{B_l}$, and a function belongs to $F_1$ if
and only if it can be written as $\varphi=\sum \alpha_i 1_{B_i}$
for some scalars $\alpha_1,\dots,\alpha_l$. The decomposition
of the function $1 \in F_1$ is given by $1= \sum 1_{B_i}$, hence
the sets $B_i$ cover the whole space up to a set of zero
measure for $\mu$. Moreover, since $B_i$ is minimal, the
measure $\displaystyle \mu_i:=\frac{1_{B_i}\mu}{\mu(B_i)}$ is an 
invariant ergodic probability measure.

Let $u_i=\Phi_1^{-1}(1_{B_i})= 1_{B_i} \pi_1 1 \in \Hcal_p^t$, then any element of $E_1$
is a linear combination of the $u_i$. In particular, this
applies to $\pi_1(fu_i)$ for any $f\in \Ccal^\alpha$ (we 
recall that $fu_i \in \Hcal_p^t$ by Lemma~\ref{lem:multiplication_smooth_functions}, 
and that $1_{B_i}$ is $T$-invariant in $\Lbb^\infty (\mu)$); we now compute $\pi_1(fu_i)$.
\begin{align*}
\langle \pi_1 (fu_i), 1_{B_j} \rangle & 
= \lim_{n \to + \infty} \frac{1}{n} \sum_{k=0}^{n-1} \int \Lcal^k (fu_i) \cdot 1_{B_j} \dL
= \lim_{n \to + \infty} \frac{1}{n} \sum_{k=0}^{n-1} \int f \cdot 1_{B_i} \cdot 1_{B_j} \circ T^k \dd \mu \\&
= \int_{B_i \cap B_j} f \dd \mu
= \mu(B_i) \delta_{i,j} \int f \dd \mu_i
\end{align*}
Hence, we get:
\begin{equation}
\label{eq:expression_fui}
\pi_1 (fu_i) = \left( \int f \dd \mu_i \right) u_i
\end{equation}

This enables us to deduce that each measure $\mu_i$ is
exponentially mixing, as follows. Let $\delta<1$ be the 
spectral radius of $\Lcal-\pi_1$; we have $\| \Lcal^n-\pi_1 \| = O(\delta^n)$. 
Then, if $f,g$ are $\Ccal^\alpha$ functions, and following \eqref{eq:expression_fui},
\begin{align*}
\int f\cdot g\circ T^n \dd\mu_i &
= \frac{1}{\mu(B_i)}\langle \Lcal^n(fu_i), g \rangle
= \frac{1}{\mu(B_i)}\langle \pi_1(fu_i), g \rangle+O(\delta^n)
\\&
= \left(\int f\dd\mu_i\right)\frac{1}{\mu(B_i)} \langle u_i, g \rangle+O(\delta^n)
\\&
= \left(\int f\dd\mu_i\right)\left(\int g\dd\mu_i\right)+O(\delta^n).
\end{align*}

Since $\Ccal^\alpha$ is dense $\Lbb^1$, every $\mu_i$ is mixing and thus ergodic.
Those $\mu_i$ are ergodic measures, and absolutely continuous with respect 
to the Lebesgue measure, so that they are also physical measures.

\smallskip
\begin{gras}Third step: basins\end{gras}

We now turn to the relationship between Lebesgue measure and
the measures $\mu_i$. For any function $f\in \Lbb^\infty(\LL)\cap
\Hcal_p^t$, let us write
\begin{equation}
\pi_1(f)=\sum_{i=1}^l a_i(f) u_i.
\end{equation}
We shall need to describe the coefficients $a_i(f)$. Let be 
$1 \leq i \leq l$, and $f \in \Lbb^\infty (\LL) \cap \Hcal_p^t$. 
We recall that the sets $B_i$ are $T$-invariant and disjoint for the Lebesgue measure 
(except perhaps for a set a zero Lebesgue measure).
\[
\lim_{n \to + \infty} \int_{X_0} f \cdot \frac{1}{n} \sum_{k=0}^{n-1} 1_{B_i} \circ T^k \dL 
= \lim_{n \to + \infty} \frac{1}{n} \sum_{k=0}^{n-1} \langle 1_{B_i}, \Lcal^k f \rangle 
= \langle 1_{B_i}, \pi_1 f \rangle
= a_i(f) \mu (B_i)
\]
Moreover, all the $1/n \sum_{k=0}^{n-1} 1_{B_i} \circ T^k$ are bounded by $1$ ; let 
$h_i$ be one of their weak $\Lbb^2 (\LL)$ limits. The function $h_i$ is 
$T$-invariant $\LL$-almost everywhere, 
and satisfies for all $f \in \Lbb^\infty (\LL) \cap \Hcal_p^t$ :
\[
a_i (f) = \frac{1}{\mu (B_i)} \int f h_i \dL
\]
Since $a_i(1)=1$, we have $\int h_i \dL=\mu(B_i)$.

Let us now compute $\int h_i h_j \dL$. To begin with, $h_j$ being invariant,
\[
a_j(1_{B_i} \circ T^n) 
= \int 1_{B_i} \circ T^n \cdot h_j \dL
= \int 1_{B_i} \circ T^n \cdot h_j \circ T^n \dL
= a_j(1_{B_i} \Lcal^n 1).
\]
The computation now gives :
\begin{align}
\label{eq:orthogonal}
\int h_i h_j \dL &
= \mu(B_j) a_j(h_i)
= \mu(B_j) \lim_{n \to + \infty} \frac{1}{n} \sum_{k=0}^{n-1} a_j(1_{B_i} \circ T^k) \\ \nonumber &
= \mu(B_j) \lim_{n \to + \infty} a_j \left( 1_{B_i} \frac{1}{n} \sum_{k=0}^{n-1} \Lcal^k 1 \right)
= \mu(B_j) a_j(u_i)
= \mu(B_i) \delta_{i,j}
\end{align}

Taking $i=j$, we get $\int h_i^2 \dL = \mu(B_i) = \int h_i \dL$. Since $h_i$ takes its
values in $[0,1]$, this shows that there exists a subset
$C_i$ of $X_0$ such that $h_i =1_{C_i}$, with
$\int 1_{C_i} \dL = \mu(B_i)$. Moreover,
\eqref{eq:orthogonal} shows that $\LL(C_i \cap C_j)=0$ if
$i \not= j$. For any function $f \in \Lbb^\infty(\LL) \cap \Hcal_p^t$,
\[
a_i(f)
= \frac{1}{\mu(B_i)} \int_{C_i} f \dL
\]
Moreover,
\[
\int_{X_0} \sum_i 1_{C_i} \dL 
= \sum_i \mu(B_i) 
= \mu(X_0) 
= \LL (X_0)
\]
This shows that the sets $C_i$ form a partition of the space
modulo a set of zero Lebesgue measure. We have proved that
\begin{equation}
\pi_1(f)=\sum_{i=1}^l \frac{1}{\mu(B_i)} \left( \int_{C_i} f \dL \right) u_i,
\end{equation}
which means that the sets $C_i$ form a partition of almost all $X_0$ 
into ergodic basins for the physical measures.
\end{proof}


\section{Piecewise $\Ccal^{1+Lip}$ maps and functions with bounded variation}
\label{sec:variation}

For now on, we will assume that the transformation $T$ is a piecewise $\Ccal^{1+Lip}$ uniformly 
expanding map, and work with the space of functions with bounded variation. 
We prove the existence of a spectral gap for the operator $\Lcal_{1/|Det DT|}$ 
in this space under simple assumptions, and compare our results with 
previous works by J.W.~Cowieson~\cite{Cowieson1} and B.~Saussol ~\cite{Saussol}.

In order to simplify the notations, in this section $\lambda_n$ denotes 
the essential infimum of the $\lambda_n (x)$.

As an introduction and a last tribute to our work with Sobolev spaces, 
we first prove a corollary to Corollary~\ref{cor:spectral_gap}.

\begin{cor}\quad
\label{cor:stupid}

If $T$ is a piecewise $\Ccal^2$ expanding map with piecewise $\Ccal^1$ boundaries and if:

\[
\lim_{n \rightarrow +\infty} (D_n^b)^\frac{1}{n} \lambda_n^{-\frac{1}{n}} < 1,
\]

then there exist $0 < t_0 < 1/p_0 < 1$ such that, 
for all $0 < t < 1/p < 1$ such that $t_0 \leq t$ and $1/p_0 \leq 1/p$, 
the essential spectral radius of $\Lcal_{1/|Det DT|}$ acting on 
$\Hcal_p^t$ is strictly smaller than $1$.
\end{cor}

\begin{proof}\quad

Let us put $\displaystyle \epsilon = 1-\lim_{n \rightarrow +\infty} 
(D_n^b)^\frac{1}{n} \lim_{n \rightarrow +\infty} \lambda_n^{-\frac{1}{n}}$.

We first choose $0< t_0 <1$ such that $\displaystyle 
\lim_{n \rightarrow +\infty} (D_n^b)^\frac{1}{n} \lim_{n \rightarrow +\infty} 
\lambda_n^{-\frac{t_0}{n}} < 1-\epsilon/2$.

Then, we choose $1< p_0 <1/t_0$ such that $\displaystyle 
\lim_{n \rightarrow +\infty} (D_n^e)^{(1-\frac{1}{p_0}) \frac{1}{n}} < 1+\epsilon/2$.

Let $1 > t \geq t_0$ and $1 < p \leq p_0$ such that $t < 1/p < 1$. The essential spectral radius 
of $\Lcal_{1/|Det DT|}$ acting on $\Hcal_p^t$ is at most 
\[
\lim_{n \to +\infty} (D_n^b)^{\frac{1}{p}\frac{1}{n}} \cdot (D_n^e)^{\frac{1}{n}(1-\frac{1}{p})} \cdot \norm{\lambda_n^{-t} |\det DT^n|^{\frac{1}{p}-1}}{\Lbb^\infty}^\frac{1}{n}
\]
\[
\leq \lim_{n \to +\infty} (D_n^b)^\frac{1}{n} \cdot (D_n^e)^{(1-\frac{1}{p_0})\frac{1}{n}} \cdot \lambda_n^{-\frac{t_0}{n}}
\leq (1-\epsilon/2)(1+\epsilon/2)<1
\]
\end{proof}

Corollary~\ref{cor:stupid} can be seen as an asymptotic result; even if we work with parameters $t$ and $p$ close
to $1$, we do not work on $\Hcal_1^1$. Indeed, this result does not hold when 
replacing $\Hcal_p^t$ by $\Hcal_1^1$. Since $T$ is piecewise continuous, we must 
allow discontinuities in the function spaces we let $\Lcal_{1/|Det DT|}$ act on, and for 
instance in dimension one every function of $\Hcal_1^1$ has a continuous version 
(because its derivative is a function in $\Lbb^1$). Here, the 
problem comes from the fact that the characteristic function of an interval is in general 
not in $H_1^1 (\R)$, since its derivative - in general a sum of atomic measures - is not in 
$\Lbb^1$. Therefore, we have to work in a broader space, for instance the space of functions 
with bounded variation; this will lead us to an alternative proof of a theorem shown 
by J.W.~Cowieson in 2000~\cite{Cowieson1}. We will also have a closer look at the results 
proved by B.~Saussol in 2000~\cite{Saussol}, which may be linked to this work in a similar fashion.


\subsection{Cowieson's theorem}

We recall here J.W. Cowieson's results~\cite{Cowieson1}, and compare them to what we prove later.

\begin{theor}[Cowieson, 2000]\quad
\label{theor:cowieson}

Let $X$ be a bounded and connected open set of $\R^d$ with piecewise $\Ccal^2$ 
boundaries, and $T : X \mapsto X$ be a piecewise $\Ccal^2$ and uniformly 
expanding map with piecewise $\Ccal^2$ boundaries.

If there exists $n \in \N$ such that $\lambda_n > D_n^b$, 
then $T$ admits an absolutely continuous invariant probability 
measure whose density is in $\BV (\overline{X})$.
\end{theor}

This theorem is indeed a consequence of Theorem~\ref{thm:BVExistSRB}. 
First, notice that $(\lambda_n)_{n \in \N}$ and $(D_n^b)_{n \in \N}$ 
are respectively supermultiplicative and submultiplicative. 
Thus, the limits $\displaystyle \lim_{n \rightarrow +\infty} 
\lambda_n^\frac{1}{n}$ and $\displaystyle \lim_{n \rightarrow +\infty} 
(D_n^b)^\frac{1}{n}$ exist, and are respectively equal to $\displaystyle 
\sup_{n \in \N} \lambda_n^\frac{1}{n}$ and $\displaystyle \inf_{n \in \N} 
(D_n^b)^\frac{1}{n}$. Therefore, $\displaystyle \lim_{n \rightarrow +\infty} 
\lambda_n^{-\frac{1}{n}} (D_n^b)^\frac{1}{n} <1$ if and only if there exists some 
$n \in \N$ such that $\lambda_n > D_n^b$, and the condition stated in 
Cowieson's theorem and condition~(\ref{eq:BVgood_condition}) are equivalent.

The differences between Cowieson's theorem and our results are the following. 
To begin with, our setting is less restrictive: the domain does not need to be 
the closure of a bounded, connected open set of some $\R^d$, but a compact subset 
of a Riemannian manifold, and the boundaries do not need to be piecewise $\Ccal^2$ 
but piecewise $\Ccal^1$. The transformation $T$ also only needs to be piecewise $\Ccal^{1+\Lip}$, 
not piecewise $\Ccal^2$. Then, the condition on the expansion rate and the combinatorial 
complexity is the same. At last, we shall prove that the 
essential spectral radius of $\Lcal_{1/|\det DT|}$ acting on $\BV$ is strictly 
smaller than $1$. Moreover, once the spectral gap is proven, we 
will get not only the existence of an absolutely continuous invariant measure 
with bounded variation, but far stronger results, as in Theorem~\ref{thm:BVExistSRB}.

We recall once more that, for generic piecewise expanding maps, the 
complexity $D^b_n$ increases subexponentially (Theorem~1.1 in \cite{Cowieson2}), so that the consequences 
of Theorem~\ref{thm:BVExistSRB} hold generically.


\subsection{Basic lemmas}

We need to adapt the lemmas from Sections~\ref{sec:inequalities} and~\ref{sec:main} to this new setting in order to get Theorem~\ref{theor:BVspectral_gap} and subsequently Theorem~\ref{thm:BVExistSRB}. They are:
\begin{itemize}
\item A Fubini-type property, which allows us to work in a $1$-dimensional setting: Lemma~\ref{lem:fubini}.
\item The continuity of the multiplication by Lipschitz functions (equivalent of Lemma~\ref{lem:multiplication_smooth_functions}): Lemma~\ref{lem:BVmultiplication_smooth_function}.
\item The continuity of the multiplication by characteristic functions of nice sets (equivalent of Lemma~\ref{lem:multiplication_charfunction}): Lemma~\ref{lem:BVmultiplication_charfunction}.
\item The effect of composition with $\Ccal^{1+\Lip}$ diffeomorphisms (equivalent of Lemma~\ref{lem:composition_diffeo}): Lemma~\ref{lem:BVcomposition_diffeo}.
\item A localization lemma (equivalent of Lemma~\ref{lem:localization}): Lemma~\ref{lem:BVlocalization}.
\item A summation lemma (equivalent to Lemma~\ref{lem:sum}).
\item The compactness of the injection from $\BV (X_0)$ to $\Lbb^1(X_0)$.
\end{itemize}

In may arguments below, we will work in fact with $\Ccal^1$ functions, 
and then proceed by approximations to get the results for functions with bounded variation, 
the approximation being done with Theorems~\ref{theor:approx1} and~\ref{theor:approx1}. 
We hope that the quite explicit computations involved will help to explain the 
properties described by the many lemmas in Section~\ref{sec:inequalities} and Section~\ref{sec:main}.

We start with the Fubini-like theorem, which allows us to reduce the 
problem to dimension $1$ in the proofs of some following lemmas. Let 
$u \in \BV(\R^d)$, $i \in 1..d$, $x \in \R^{d-1}$ and $t \in \R$.

We put $x = (x_1,...,x_{i-1},x_{i+1},...,x_d)$. 
We will denote by $u_{x,i}$ the map $t \mapsto u(x_1,...,x_{i-1},t,x_{i+1},...,x_d)$, 
and by $u_i$ the map $x \mapsto \| u_{x,i} \|_{\BV(\R)}$.

\begin{lem}[Fubini-type property]\quad
\label{lem:fubini}

For all $u \in \Ccal^1 \cap \BV (\R^d)$ and $i \in 1..d$, then $u_i \in \Lbb^1$ and
\[
\norm{u}{\BV (\R^d)} 
= d^{-1}\sum_{i=1}^d \norm{u_i}{\Lbb^1}
\]
\end{lem}

\begin{proof}\quad

Let $u \in \Ccal^1$. Fubini-Tonelli theorem gives, 
for all $i \in 1..d$, $\norm{u}{\Lbb^1} = \norm{\norm{u_{.,i}}{\Lbb^1}}{\Lbb^1}$, hence
\[
\norm{u}{\Lbb^1} 
= d^{-1} \sum_{i=1}^d \norm{\norm{u_{.,i}}{\Lbb^1}}{\Lbb^1} 
\]

On the other hand,
\[
V(u) = d^{-1}\int_{\R^d} \sum_{i=1}^d \left| \frac{\partial u}{\partial x_i} \right| \dL
= d^{-1}\sum_{i=1}^d \int_{\R^{d-1}} V(u_{.,i}) \dL
= d^{-1}\sum_{i=1}^d \norm{V(u_{.,i})}{\Lbb^1}.
\]

Then, all is left is to sum the equalities.
\end{proof}

For any Lipschitz function $u \in \Lip(\R^d,\C)$, we put $\norm{u}{\Lip} = \norm{u}{\Lbb^\infty} + L(u)$, 
where $L(u)$ is the best Lipschitz constant for $u$. With this lemma, we can easily deduce the following:

\begin{lem}[Multiplication by Lipschitz functions]\quad
\label{lem:BVmultiplication_smooth_function}

For all $u \in \BV (\R^d)$, for all $g \in \Lip(\R^d,\C)$,
\begin{equation*}
\norm{gu}{\BV(\R^d)}
\leq \norm{g}{\Lip} \norm{u}{\BV(\R^d)}
\end{equation*}
\end{lem}

\begin{proof}\quad

Assume first that $u \in \Ccal^1$. Thanks to the Fubini-type property of Lemma~\ref{lem:fubini}, 
we just have to show that $\norm{gu}{\BV(\R)} \leq \norm{g}{\Lip} \norm{u}{\BV(\R)}$. 
Since $g$ is continuous, this will be an application of the Leibniz formula 
(see e.g. Proposition~3.2 in \cite{Baladi1}; the derivative of $g$ has to be taken in a weak sense):

\begin{align*}
\norm{gu}{\BV(\R)} &
= \int_\R |gu| \dL + \int_\R |D(gu)| \dL \\&
\leq \int_\R |gu| \dL + \int_\R |gDu| \dL + \int_\R |uDg| \\&
\leq \norm{g}{\Lbb^\infty} \int_\R |u| \dL + \norm{g}{\Lbb^\infty} \int_\R |Du| \dL + \norm{Dg}{\Lbb^\infty} \int_\R |u| \dL \\
\leq \norm{g}{\Lip} \norm{u}{\BV(\R)}
\end{align*}

For any $u \in \BV(\R)$, let be a sequence of $\Ccal^1$ functions $(u_n)_{n \in \N}$ 
approaching $u$ in the sense of Theorem~\ref{theor:approx1}.

Then, since 
\[
\norm{gu_n-gu}{\Lbb^1} \leq \norm{g}{\Lbb^\infty}\norm{u_n-u}{\Lbb^1},
\]
we see that $(gu_n)$ is a sequence of $\BV(\R)$ functions which converges to $gu$ in $\Lbb^1$. 
Theorem~\ref{theor:approx2} tells us that:
\[
\norm{gu}{\BV(\R)}
\leq \liminf_{n \to +\infty} \norm{gu_n}{\BV(\R)} 
\leq \norm{g}{\Lip} \liminf_{n \to +\infty} \norm{u_n}{\BV(\R)} 
\leq \norm{g}{\Lip} \norm{u}{\BV(\R)}
\]
\end{proof}

\begin{lem}[Multiplication by characteristic functions of nice sets]\quad
\label{lem:BVmultiplication_charfunction}

Let $O$ be a measurable subset of $\R^d$ whose intersection 
with almost every line parallel to a coordinate axis has at most $L$ 
connected components. Then, for all $u \in \BV$,

\begin{equation*}
\label{eq:premiere_inegalite_charfunc}
\norm{1_O u}{\BV} 
\leq L \norm{u}{\BV}.
\end{equation*}
\end{lem}

\begin{proof}\quad

Assume that $u \in \Ccal^1$. Using again the Fubini-type property of Lemma~\ref{lem:fubini}, we just have 
to show that $\norm{1_O u}{\BV(\R)} \leq L \norm{u}{\BV(\R)}$ if 
$O$ is an union of intervals, or even that $\norm{1_O u}{\BV(\R)} 
\leq \norm{u}{\BV(\R)}$ if $O$ is an interval, and then sum the 
contributions of the different intervals. Let us put $O = [a,b]$ 
(the case $O = \R_+$ is even easier). Then:

\[
\norm{1_{O}u}{\Lbb^1} \leq \norm{u}{\Lbb^1},
\]
\begin{align*}
V(1_{O} u) &
= |u(a)| + \int_{[a,b]} |u'(x)| \dd x + |u(b)| \\&
= \left| \int_{(-\infty,a]} u'(x) \dd x \right| + \int_{[a,b]} |u'(x)| \dd x + \left| \int_{[b,+\infty)} u'(x) \dd x \right|
\leq V(u).
\end{align*}

Summing those inequalities ends this proof for $\Ccal^1$ functions. For a general function 
in $\BV(\R)$, and using the same trick as in the end of the proof of Lemma~\ref{lem:BVmultiplication_smooth_function}, 
we just have to show that, if $(u_n)$ converges to $u$ in $\Lbb^1$, then $(1_O u_n)$ converges 
to $1_O u$ in $\Lbb^1$, which is trivial.
\end{proof}

The next lemma deals with the composition with $\Ccal^{1+\Lip}$ diffeomorphisms.

\begin{lem}[Composition with $\Ccal^{1+\Lip}$ diffeomorphisms]\quad
\label{lem:BVcomposition_diffeo}

Let $F \in \Ccal^{1+\Lip} (\R^d,\R^d)$ be a diffeomorphism 
and let $A \in GL_d (\R)$ such that, for all $z \in \R^d$, 
$\| A^{-1} \circ DF(z) \| \leq 2$ and $\| DF(z)^{-1} \circ A \| \leq 2$. 
Then, for all $u \in \Ccal^1 (\R^d)$,

\begin{equation*}
\norm{u \circ F}{\BV(\R^d)} 
\leq 2^{d+1} | \det A |^{-1} \|A\| \norm{u}{\BV(\R^d)} + 2^{d+1} | \det A |^{-1} \norm{u}{\Lbb^1}.
\end{equation*}
\end{lem}

\begin{proof}\quad

Assume first that $u$ is $\Ccal^1$. We put, as in the proof of 
Lemma~\ref{lem:composition_diffeo}, $u \circ F = u \circ A \circ A^{-1} \circ F$. 
Let us deal with $u \circ A$ first.

We have $\norm{u \circ A}{\Lbb^1} \leq  | \det A |^{-1} \norm{u}{\Lbb^1}$ and $V(u \circ A) \leq | \det A |^{-1} \| A \| V(u)$, so that

\begin{equation}
\label{eq:BVpremiere_inegalite_diffeo}
\norm{u \circ A}{\BV(\R^d)} \leq  | \det A |^{-1} \|A\| \norm{u}{\BV(\R^d)} + | \det A |^{-1} \norm{u}{\Lbb^1}.
\end{equation}

We now deal with $u \circ A^{-1} \circ F$; we write $\tilde{F} 
= A^{-1} \circ F$. Since the derivative of $\tilde{F}$ is everywhere bounded 
between $1/2$ and $2$, $| \det D (\tilde{F}^{-1}) | \leq 2^d$ and a 
change of variables gives $\norm{u \circ \tilde{F}}{\Lbb^1} 
\leq  2^d \norm{u}{\Lbb^1}$ and $V(u \circ \tilde{F}) \leq 
2^{d+1} V(u)$, so that

\begin{equation}
\label{eq:BVseconde_inegalite_diffeo}
\norm{u \circ \tilde{F}}{\BV(\R^d)} \leq  2^{d+1} \norm{u}{\BV(\R^d)}
\end{equation}

Together, \eqref{eq:BVpremiere_inegalite_diffeo} and 
\eqref{eq:BVseconde_inegalite_diffeo} give 
\[
\norm{u \circ F}{\BV(\R^d)}
\leq 2^{d+1} | \det A |^{-1} \|A\| \norm{u}{\BV(\R^d)} + 2^{d+1} 
| \det A |^{-1} \norm{u}{\Lbb^1}.
\]

Now, let us take any function $u$ in $\BV(\R^d)$, and use the same trick as in the end of the proof of 
Lemma~\ref{lem:BVmultiplication_smooth_function}. Since $\norm{u \circ \tilde{F}}{\Lbb^1} 
\leq  2^d \norm{u}{\Lbb^1}$ holds for any $\Lbb^1$ function, if $(u_n)$ converges to 
$u$ in $\Lbb^1$, then $(u_n \circ \tilde{F})$ converges 
to $u \circ \tilde{F}$ in $\Lbb^1$. Hence, inequality 
\eqref{eq:BVseconde_inegalite_diffeo} holds for any function in $\BV(\R^d)$.

Inequality \eqref{eq:BVpremiere_inegalite_diffeo} can be obtained with a slighly different 
method: $\norm{u \circ A}{\Lbb^1} \leq  | \det A |^{-1} \norm{u}{\Lbb^1}$ holds obviously 
for any $u$ in $\Lbb^1$. We just have to prove that $V(u \circ A) \leq | \det A |^{-1} \| A \| V(u)$ 
for any $u$ in $\BV (\R^d)$; the approximation by $\Ccal^1$ function still works (the reader can 
easily check that Theorem~\ref{theor:approx2} is still true if one writes $V(\cdot)$ instead of $\norm{\cdot}{\BV}$).

Therefore, Lemma~\ref{lem:BVcomposition_diffeo} holds for any function in $\BV(\R^d)$.
\end{proof}

A consequence of \eqref{eq:BVpremiere_inegalite_diffeo} is that, 
for all $u \in \BV(\R^d)$ and $A \in GL_d (\R)$, $\| u \circ A \|_{\BV} 
\leq 2 \max (\|A\|, 1) | \det A |^{-1} \| u \|_{\BV}$.

The next result to prove is the equivalent of the first localization principle, Lemma~\ref{lem:localization}.

\begin{lem}[Localization principle]\quad
\label{lem:BVlocalization}

Let be $\eta \in \Ccal_c^1 (\R^d,\R)$, and write, for all $x \in \R^d$ 
and $m \in \Z^d$, $\eta_m (x) = \eta (x+m)$. Then, there exists a 
constant $C_\eta$ such that, for all $u \in \BV (\R^d)$,

\[
\sum_{m \in \Z^d} \norm{\eta_m u}{\BV} \leq C_\eta \norm{u}{\BV}.
\]
\end{lem}

\begin{proof}\quad

We define on $\R^d$ a relation by $x<y$ if and only if $x_i<y_i$ 
for all $i \in \{ 1,...,d \}$, and another relation $\leq$ the 
same way. For $k \in \Z^d$, $0<k$, we define $O_k = \{ x \in \R^d : 
0<x<k\}$. Up to a translation, we may assume that there exists some  $k 
= (k_1,...,k_d) \in \Z^d$ such that $\Supp(\eta) \subset O_k$. 
We also note, for $\lambda \in \Z^d$, by $\lambda k$ the vector 
$(\lambda_1 k_1,..,\lambda_d k_d)$.

We point out that, for any $u$ in $\BV (\R^d)$ and $\lambda$ in $\R^d$, the 
sum $\displaystyle \sum_{\lambda \in \Z^d} \norm{u}{\BV(\lambda k +l+ O_k)}$ 
is equal to $\displaystyle \norm{u}{\BV(\bigcup_{\lambda \in \Z^d} \lambda k +l+ O_k)}$, 
which is at most $\displaystyle \norm{u}{\BV (\R^d)}$ (the $\Lbb^1$ norms are the same, 
and the inequality of the variations comes directly from Definition~\ref{def:BV})

We now split the sum, and then apply Lemma~\ref{lem:BVmultiplication_smooth_function}:

\begin{align*}
\sum_{m \in \Z^d} \norm{\eta_m u}{\BV(\R^d)} &
= \sum_{\lambda \in \Z^d} \sum_{\substack{l \in \Z^d \\ 0 \leq l < k}} \norm{\eta_{\lambda k + l} u}{\BV(\lambda k + l + O_k)} \\&
\leq \norm{\eta}{\Ccal^1} \sum_{\substack{l \in \Z^d \\ 0 \leq l < k}} \sum_{\lambda \in \Z^d} \norm{u}{\BV(\lambda k +l+ O_k)} \\&
\leq \norm{\eta}{\Ccal^1} \sum_{\substack{l \in \Z^d \\ 0 \leq l < k}} \norm{u}{\BV(\R^d)} \\&
= \norm{\eta}{\Ccal^1} \left( \prod_{i=1}^d k_i \right) \norm{u}{\BV(\R^d)}.
\end{align*}

This is the lemma, with $\displaystyle C_\eta = \norm{\eta}{\Ccal^1} \prod_{i=1}^d k_i$.
\end{proof}

Since the space $\BV$ make us work morally with $p=1$ and $t=1$, 
an equivalent of Lemma~\ref{lem:sum} for functions with bounded variation is that, for
any functions $v_1,\dots, v_l$ with compact support in $\R^d$ and belonging to $\BV(\R^d)$,

\begin{equation*}
\norm{ \sum_{i=1}^l v_i}{\BV(\R^d)}
\leq \sum_{i=1}^l \norm{v_i}{\BV(\R^d)}.
\end{equation*}

This is just the triangular inequality for the $\BV (\R^d)$ norm, so that there is nothing to prove here.

Finally, it is a well-known fact that, for any open and bounded set with nice boundaries 
(for instance piecewise $\Ccal^1$ boundaries) $\Omega$ of some $\R^d$, the canonical immersion of $\BV(\Omega)$ 
into $\Lbb^1(\Omega)$ is compact; see for instance Theorem~1.19 in \cite{Giusti}. 
The compactness of the canonical immersion of $\BV(X_0)$ into $\Lbb^1(X_0)$ ensues.


\subsection{Saussol's theorem}

In sections~6.1 and~6.2, we have seen how to obtain Theorem~\ref{thm:BVExistSRB}, 
that is, a version of Theorem~\ref{thm:ExistSRB} when $T$ is a piecewise $\Ccal^{1+\Lip}$ 
uniformly expanding map, in the limit $p=1$ and $t=1$. This involves the space of functions 
with bounded variation, since the Sobolev space $\Hcal_1^1$ is not suitable for this task (for 
instance, it is not large enough in dimension $1$). One might want to know suitable spaces when 
one works instead with a piecewise $\Ccal^{1+\alpha}$ uniformly expanding map, and wants to get to the limit $t=\alpha$ 
(in this setting, Theorem~\ref{theor:spectral_gap} tells nothing about the essential spectral radius of the 
Perron-Frobenius operator on the $\Hcal_p^\alpha$ spaces).

The last part of this article is a study of a previous result by B.~Saussol~\cite{Saussol}, 
which is very close to our own results and gives us another example of function spaces to which our 
method could be applied, with the parameters $t=\alpha$ and $p=1$. B.~Saussol proved the existence of a spectral gap of 
$\Lcal_{1/|\det DT|}$ when acting on spaces of functions with bounded 
oscillation $\V_\alpha$, when $T$ is a piecewise $\Ccal^{1+\alpha}$ uniformly expanding map. 

First, let us present the functions with bounded oscillation. Let $X_0$ be a compact set of some 
$\R^d$. For any Borel set $S \in \Bcal (\R^d)$ such that $\LL (S) > 0$, and any $f \in \Lbb^1 (\R^d)$, 
we denote the essential infimum of $f$ on $S$ by $\Einf_S f$, and its essential supremum by $\Esup_S f$. 
Next, we choose some $\epsilon_0 > 0$ and $\alpha \in (0,1]$, and define for all $f \in \Lbb^1 (\R^d)$:

\begin{equation}
\osc(f,S) = \Esup_S f - \Einf_S f
\end{equation}
\begin{equation}
|f|_\alpha = \sup_{0<\epsilon \leq \epsilon_0} \epsilon^{-\alpha} \int_{\R^d} \osc(f,B(x,\epsilon)) \dd x
\end{equation}

Then, we define $\V_\alpha (\R^d) = \{ f \in \Lbb^1 (\R^d):|f|_\alpha < + \infty \}$ and 
$\V_\alpha = \V_\alpha (X_0) = \{ f \in \V_\alpha (\R^d)):\Supp f \subset X_0 \}$, both endowed with the norm 
$\|f\|_{\V_\alpha} = \|f\|_{\Lbb^1} + |f|_\alpha$.

\begin{rem}\quad

Different choices of $\epsilon_0$ lead to different norms, although they are all equivalent: 
the space $\V_\alpha$ does not depend on $\epsilon_0$.

We have already encountered $\V_1 (\R)$: it is the same space as $\BV (\R)$ \cite{Keller3}. 
Hence, the results of Saussol are a generalization of the previous theorem by A.~Lasota and J.A.~Yorke \cite{Lasota}.
\end{rem}

\begin{rem}\quad

We endow $\Ccal_\alpha$, the set of $\alpha$-H\"{o}lder functions on $X_0$, with the norm 
$\|f\|_\alpha = \|f\|_{\Lbb^1} + |f|'_\alpha$, where:

\[
|f|'_\alpha = \sup_{x,y \in X_0} \frac{|f(x)-f(y)|}{|x-y|^\alpha}
\]

For $f \in \Ccal_\alpha$, $x \in X_0$ and $\epsilon > 0$, we get easily the following inequalities: 

\begin{align*}
|f|_\alpha &
= \sup_{0<\epsilon \leq \epsilon_0} \epsilon^{-\alpha} \int_{\R^d} \osc(f,B(x,\epsilon)) \dd x \\&
\leq \sup_{0<\epsilon \leq \epsilon_0} \epsilon^{-\alpha} (2 \epsilon)^\alpha \LL (X_0) |f|'_\alpha \\&
\leq 2^\alpha \LL (X_0) |f|'_\alpha
\end{align*}

Hence, for all $\alpha \in (0,1]$, we have a continuous inclusion from $\Ccal_\alpha$ into $\V_\alpha$. 
However, $\V_\alpha$ is much larger, so that functions belonging to this space may present discontinuities.
\end{rem}

The setting studied by B.~Saussol is more general than the one we presented in Section~\ref{sec:setting}: 
for instance, it allows under some conditions boundaries whose Hausdorff dimension is strictly larger 
than $d-1$, or maps whose sets of continuity are countably many. However, one of these conditions is very 
abstract, and the use of more flexible conditions leads naturally to our setting.

In the following, we consider that $\alpha \in (0,1]$ is fixed, and we put $\gamma_d = \LL (B(0,1))$. 
Here is the main theorem, an adaptation of Theorem~5.1 and Lemma~2.1 in \cite{Saussol} with the additional use of 
Hennion's theorem~\cite{Hennion}):

\begin{theor}[Saussol's theorem]\quad
\label{theor:Saussol}

Let $T$ be a piecewise $\Ccal^{1+\alpha}$ uniformly expanding map. $\Lcal_{1/|\det DT|}$ 
acts continuously on $\V_\alpha$, and its essential spectral radius is at most:

\begin{equation}
\label{eq:Vgood_condition}
\lambda^{-\alpha} + \frac{4 \gamma_d D_1^b}{(\lambda-1) \gamma_{d-1}}.
\end{equation}
\end{theor}

This theorem naturally leads to a result of existence of physical measures 
(stated in a different way in \cite{Saussol}) similar to Theorem~\ref{thm:ExistSRB}:

\begin{theor}\quad
\label{thm:VExistSRB}

Let $t$ be a piecewise $\Ccal^{1+\alpha}$ uniformly expanding map such that the 
bound \ref{eq:Vgood_condition} for the essential spectral radius is smaller than $1$.
Then $T$ has a finite number of physical measures whose densities are in $\V_\alpha$, 
which are ergodic, and whose basins cover Lebesgue almost all $X_0$.
Moreover, if $\mu$ is one of these measures, there exist an
integer $k$ and a decomposition $\mu=\mu_1+\dots+\mu_k$ such
that $T$ sends $\mu_j$ to $\mu_{j+1}$ for $j\in \Z/k\Z$, and
the probability measures $k\mu_j$ are mixing at an exponential rate for
$T^k$ and $\alpha$-H\"{o}lder test functions.
\end{theor}

Since the setting and conclusion are already familiar to the reader, we will look more 
closely at the upper bound of the essential spectral radius.
The estimate \ref{eq:Vgood_condition} is rather rough, and for $\alpha=1$ and small $d$ 
is worse than the one given in Theorem~\ref{thm:BVExistSRB}. However, just like in the 
proof of Theorem~\ref{theor:spectral_gap}, one may iterate the transformation to get 
better estimates of the essential spectral radius $\rho_{ess}$ of the Perron-Frobenius 
operator, and then take the limit as the number of iterations grows to infinity:

\begin{equation}
\label{eq:Vspectral_gap}
\rho_{ess} 
\leq \sup \left\{ \lim_{n \to + \infty} \lambda_n^{-\frac{1}{n}} (D_n^b)^{\frac{1}{n}}; \lim_{n \to + \infty} \lambda_n^{-\frac{\alpha}{n}} \right\}
\end{equation}

\begin{proof}\quad

The expression \ref{eq:Vgood_condition} appears in a Lasota-Yorke type inequality in the proof of 
Saussol's theorem. When using Hennions's theorem~\cite{Hennion}, one can get a better bound on the essential spectral 
radius of an operator by iterating this operator. In other words, we have for all positive integer $n$:

\[
\rho_{ess} 
\leq \left(\lambda_n^{-\alpha} + \frac{4 \gamma_d D_n^b}{(\lambda_n-1) \gamma_{d-1}}\right)^\frac{1}{n}.
\]

Let us put $\displaystyle C=\sup \left\{ \lim_{n \to + \infty} \lambda_n^{-\frac{1}{n}} (D_n^b)^{\frac{1}{n}}; 
\lim_{n \to + \infty} \lambda_n^{-\frac{\alpha}{n}} \right\}$. Let $\epsilon >0$. Obviously,

\[
\lim_{n \to + \infty} \left(\lambda_n^{-\alpha} + \frac{4 \gamma_d D_n^b}{(\lambda_n-1) \gamma_{d-1}}\right) 
(C+\epsilon)^{-n} = 0.
\]

Hence,

\[
\limsup_{n \to + \infty} \left(\lambda_n^{-\alpha} + \frac{4 \gamma_d D_n^b}{(\lambda_n-1) \gamma_{d-1}}\right)^\frac{1}{n} (C+\epsilon)^{-1} \leq 1.
\]

Since this is true for all $\epsilon >0$, inequality~\ref{eq:Vspectral_gap} ensues.
\end{proof}

Clearly, a sufficient condition for the conclusions of Theorem~\ref{thm:VExistSRB} to 
hold is $\lim_{n \to + \infty} \lambda_n^{-\frac{1}{n}} (D_n^b)^{\frac{1}{n}}<1$.

The $\V_\alpha$ spaces were constructed to satisfy the same properties that we needed 
(they include the characteristic functions of nice enough sets, the multiplication 
by $\alpha$-H\"{o}lder functions behaves nicely, and the injection into $\Lbb^1$ is compact), 
and one should be able to adapt the different lemmas as we did for the space of functions with bounded variation. 
Finally, our method would probably give a different (and worse) estimate of the essential spectral radius of the Perron-Frobenius operator, 
such as (we take $t=\alpha$ and $p=1$ in Theorem~\ref{theor:spectral_gap}):

\[
\rho_{ess} 
\leq \lim_{n \to +\infty}
(D_n^b)^\frac{1}{n} \cdot
\lambda_n^{-\frac{\alpha}{n}}
\]

B. Saussol also gives a lower bound on the spectral gap via the study of the cones in $\V_\alpha$ 
(for another example of cone contraction method, see e.g. \cite{Baladi1}), and an upper bound on 
the number of ergodic physical measures. Such features could perhaps be adapted to our current setting, 
with Sobolev spaces or the space of functions with bounded variation.


\addcontentsline{toc}{section}{Bibliography}
\begin{thebibliography}{BKL02}

\bibitem[Bal]{Baladi2}
V.~Baladi.
\newblock {\em Anisotropic Sobolev spaces and dynamical transfer operators:
  $\Ccal^\infty$ foliations}.
\newblock {{Algebraic and topological dynamics}}, Contemporary Mathematics, 123-136, 2005

\bibitem[Bal00]{Baladi1}
V.~Baladi.
\newblock {\em {Positive transfer operators and decay of correlations}}.
\newblock World scientific, 2000.

\bibitem[BG09]{BaladiHyp1}
V.~Baladi and S.~Gouëzel.
\newblock {\em Good Banach spaces for piecewise hyperbolic maps via interpolation}.
\newblock {Annales de l'Institut Henri Poincaré, Analyse non linéaire},
  26:1453--1481, 2009.

\bibitem[BG10]{BaladiHyp2}
V.~Baladi and S.~Gouëzel.
\newblock {\em Banach spaces for piecewise cone hyperbolic maps}.
\newblock {Journal of Modern Dynamics}, 4:91--137, 2010.

\bibitem[BKL02]{Blank}
M.~Blank, G.~Keller, and C.~Liverani.
\newblock {\em Ruelle-Perron-Frobenius spectrum for Anosov maps}.
\newblock {Nonlinearity}, 15:1905--1973, 2002.

\bibitem[Buz97]{Buzzi3}
J.~Buzzi.
\newblock {\em Intrisic ergodicity of piecewise affine maps in $[0,1]^d$}.
\newblock {Monatshefte für Mathematik}, 124:97--118, 1997.

\bibitem[Buz01]{Buzzi1}
J.~Buzzi.
\newblock {\em No or infinitely many A.C.I.P. for piecewise expanding $\Ccal^r$
  maps in higher dimensions}.
\newblock {Communications in Mathematical Physics}, 222:495--501, 2001.

\bibitem[Cow00]{Cowieson1}
W.J. Cowieson.
\newblock {\em Stochastic stability for piecewise expanding maps in $\R^d$}.
\newblock {Nonlinearity}, 13:1745--1760, 2000.

\bibitem[Cow02]{Cowieson2}
W.J. Cowieson.
\newblock {\em Absolutely continuous invariant measures for most piecewise smooth
  expanding maps}.
\newblock {Ergodic Theory and Dynamical Systems}, 22:1061--1078, 2002.

\bibitem[GB89]{Gora1}
P.~G\'{o}ra and A.~Boyarsky.
\newblock {\em Absolutely continuous invariant measures for piecewise expanding
  $\Ccal^2$ transformations in $\R^n$}.
\newblock {Israël Journal of Mathematics}, 67:272--290, 1989.

\bibitem[Giu84]{Giusti}
E.~Giusti.
\newblock {\em {Minimal surfaces and functions of bounded variation}}.
\newblock Birkhaüser, Boston, 1984.

\bibitem[Hen93]{Hennion}
H.~Hennion.
\newblock {\em Sur un théorème spectral et son application aux noyaux
  lipchitziens}.
\newblock {Proceedings of the American Mathematical Society}, 118:627--639,
  1993.

\bibitem[Kel79]{Keller1}
G.~Keller.
\newblock {\em Ergodicité et mesures invariantes pour les transformations
  dilatantes par morceau d'une région bornée du plan}.
\newblock {Comptes-rendus de l'Académie des Sciences de Paris},
  289:625--627, 1979.

\bibitem[Kel85]{Keller3}
G.~Keller.
\newblock {\em Generalized bounded variation and applications to piecewise
  monotonic transformations}.
\newblock {Zeitschrift für Wahrscheinlichkeitheorie und verwandte
  Geliete}, 69:461--478, 1985.

\bibitem[KR04]{Keller2}
G.~Keller and H.H. Rugh.
\newblock {\em Eigenfunctions for smooth expanding circle maps}.
\newblock {Nonlinearity}, 17:1723--1730, 2004.

\bibitem[LY73]{Lasota}
A.~Lasota and J.A. Yorke.
\newblock {\em On the existence of invariant measures for piecewise monotonic
  transformations}.
\newblock {Transactions of the American Mathematical Society},
  186:481--488, 1973.

\bibitem[Sau00]{Saussol}
B.~Saussol.
\newblock {\em Absolutely continuous invariant measures for multidimensional
  expanding maps}.
\newblock {Israël Journal of Mathematics}, 116:223--248, 2000.

\bibitem[Str67]{Strichartz}
R.S. Strichartz.
\newblock {\em Multipliers on fractional Sobolev spaces}.
\newblock {Journal of Mathematics and Mechanics}, 16:1031--1060, 1967.

\bibitem[Tri77]{Triebel3}
H.~Triebel.
\newblock {\em General function spaces. III. (Spaces $B_{p,q}^{g(x)}$ and
  $F_{p,q}^{g(x)}$, $1<p<\infty$: basic properties.)}.
\newblock {Analysis Mathematica}, 3:221--249, 1977.

\bibitem[Tri78]{Triebel1}
H.~Triebel.
\newblock {\em {Interpolation theory, function spaces, differential
  operators}}.
\newblock North-Holland, Amsterdam, 1978.

\bibitem[Tri92]{Triebel2}
H.~Triebel.
\newblock {\em {Theory of function spaces II}}.
\newblock Birkhäuser, Basel, 1992.

\bibitem[Tsu00]{Tsujii}
M.~Tsujii.
\newblock {\em Piecewise expanding maps on the plane with singular ergodic
  properties}.
\newblock {Ergodic Theory and Dynamical Systems}, 20:1851--1857, 2000.

\end{thebibliography}
\end{document}